\documentclass[a4paper,12pt]{amsart}

\usepackage[english]{babel}
\usepackage{amsmath,amsthm,amssymb,amsfonts}
\usepackage[all]{xy}
\usepackage{multicol}
\usepackage{todonotes}
 \usepackage{enumitem} 
\usepackage[dvipsnames,table]{xcolor}
\usepackage{pdflscape}
\usepackage[bookmarks=true,hyperindex,pdftex,colorlinks,citecolor=cyan]{hyperref}
\usepackage{mathtools}
\usepackage{extarrows}
\usepackage{soul}

\hypersetup{colorlinks=true,  linkcolor=blue, citecolor=red, urlcolor=cyan}
\usepackage{cleveref}
\usepackage[export]{adjustbox}
\usepackage{graphicx}
\usepackage{subcaption}
\usepackage{tikz}
\usepackage{tikz-cd}
\usetikzlibrary{arrows.meta,calc,matrix}
\usetikzlibrary{arrows}
\usetikzlibrary{arrows.meta, positioning}
\usepackage{caption}
\usepackage[normalem]{ulem}
\usepackage{marginnote}
\usepackage{cancel}

\usepackage{lscape}

\setlist[enumerate]{itemsep=0.2em, topsep=0.25em}

\makeatletter
\newcommand*\bigcdot{\mathpalette\bigcdot@{.5}}
\newcommand*\bigcdot@[2]{\mathbin{\vcenter{\hbox{\scalebox{#2}{$\m@th#1\bullet$}}}}}
\makeatother
\newcommand{\sslash}{\mathbin{/\mkern-6mu/}}
 \newcommand{\vertiii}[1]{{\left\vert\kern-0.25ex\left\vert\kern-0.25ex\left\vert #1 
		\right\vert\kern-0.25ex\right\vert\kern-0.25ex\right\vert}}

\usepackage{tikz}
\usetikzlibrary{mindmap,backgrounds, calc}
\usetikzlibrary{matrix,chains,scopes,positioning,arrows,fit}
\usetikzlibrary{positioning, shapes, patterns}
\usetikzlibrary{automata} 
\usetikzlibrary{shapes.geometric, arrows, arrows.meta}
\usetikzlibrary{positioning,decorations.markings}

\usepackage[a4paper,left=3cm,right=3cm,top=2.5cm,bottom=3cm]{geometry}

\theoremstyle{definition}
\newcounter{maincoro}

\newtheorem{theorem}{Theorem}[section]
\newtheorem{lemma}[theorem]{Lemma}
\newtheorem{fact}[theorem]{Fact}

\newtheorem{proposition}[theorem]{Proposition}
\newtheorem{corollary}[theorem]{Corollary}

\theoremstyle{definition}
\newcounter{maintheorem}

\newtheorem{definition}[theorem]{Definition}
\newtheorem{example}[theorem]{Example}
\newtheorem{problem}[theorem]{Problem}

\theoremstyle{remark}
\newtheorem{remark}[theorem]{Remark}
\numberwithin{equation}{section}

\newcommand{\R}{\mathbb{R}}

\newcommand{\N}{\mathbb{N}}

\newcommand{\FF}{\mathcal{F}}
\newcommand{\A}{\mathcal{A}}

\renewcommand{\d}{{\rm d}}
\newcommand{\clco}{\mathop{\overline{\mathrm{co}}}\nolimits}
\newcommand{\Span}{\mathrm{span}}
\newcommand{\Iso}{\mathrm{Iso}}
\newcommand{\Aut}{\mathrm{Aut}}

 \makeatletter
\tikzset{
  arrow/.style   = {-{Implies}, double equal sign distance, line width=0.65pt},
  iffarrow/.style= {{Implies}-{Implies}, double equal sign distance, line width=0.65pt},
}

\renewcommand{\tocsection}[3]{%
	\indentlabel{\@ifnotempty{#2}{\bfseries\ignorespaces#1 #2\quad}}\bfseries#3}
\renewcommand{\tocsubsection}[3]{%
	\indentlabel{\@ifnotempty{#2}{\ignorespaces#1 #2\quad}}#3}

\def\@tocline#1#2#3#4#5#6#7{\relax
	\ifnum #1>\c@tocdepth 
	\else
	\par \addpenalty\@secpenalty\addvspace{#2}%
	\begingroup \hyphenpenalty\@M
	\@ifempty{#4}{%
		\@tempdima\csname r@tocindent\number#1\endcsname\relax
	}{%
		\@tempdima#4\relax
	}%
	\parindent\z@ \leftskip#3\relax \advance\leftskip\@tempdima\relax
	\rightskip\@pnumwidth plus1em \parfillskip-\@pnumwidth
	#5\leavevmode\hskip-\@tempdima{#6}\nobreak
	\leaders\hbox{$\m@th\mkern \@dotsep mu\hbox{.}\mkern \@dotsep mu$}\hfill
	\nobreak
	\hbox to\@pnumwidth{\@tocpagenum{\ifnum#1=1\bfseries\fi#7}}\par
	\nobreak
	\endgroup
	\fi}
\AtBeginDocument{%
	\expandafter\renewcommand\csname r@tocindent0\endcsname{0pt}
}
\def\l@subsection{\@tocline{2}{0pt}{2.5pc}{5pc}{}}
\makeatother

\DeclareMathOperator{\dist}{dist\,}
\DeclareMathOperator{\diam}{diam\,}
\DeclareMathOperator{\co}{co}

\DeclareMathOperator{\id}{Id}

\DeclareMathOperator{\iso}{Iso}

\newcommand{\nn}[1]{{\left\vert\kern-0.25ex\left\vert\kern-0.25ex\left\vert #1 
		\right\vert\kern-0.25ex\right\vert\kern-0.25ex\right\vert}}

\renewcommand{\geq}{\geqslant}
\renewcommand{\leq}{\leqslant}

\newcommand{\restricted}{\mathord{\upharpoonright}}
\newcommand{\norm}[1]{\left\Vert#1\right\Vert}
\newcommand{\abs}[1]{\left\vert#1\right\vert}

\newcommand{\NA}{\operatorname{NA}}

\newcommand{\e}{\varepsilon}

\newcommand{\ext}{\operatorname{ext}}

\newcommand{\eps}{\varepsilon}

\begin{document}

\title[Group equivariant RNP and its characterizations]{Group equivariant Radon-Nikod\' ym property and its characterizations}

\author[S.~Dantas]{Sheldon Dantas}
\address[S.~Dantas]{Czech Technical University in Prague, FEE, Department of Mathematics, Technická 2, 16627, Prague 6, Czech Republic \newline
\href{https://orcid.org/0000-0001-8117-3760}{ORCID: \texttt{0000-0001-8117-3760}}}
\email{\texttt{sheldon.dantas@fel.cvut.cz}}
\urladdr{www.sheldondantas.com}

\author[Doucha]{Michal Doucha}
\address[Doucha]{Institute of Mathematics, Czech Academy of Sciences, Žitná 25, 115 67 Praha 1, Czechia\newline
\href{https://orcid.org/0000-0003-3675-1378}{ORCID: \texttt{0000-0003-3675-1378}}}
\email{doucha@math.cas.cz}
\urladdr{https://users.math.cas.cz/~doucha/}

\author[Jung]{Mingu Jung}
        \address[Jung]{Department of Mathematics \& Research Institute for Natural Sciences, Hanyang University, 04763 Seoul, Republic of Korea \newline
\href{https://orcid.org/0000-0003-2240-2855}{ORCID: \texttt{0000-0003-2240-2855}}}
\email{mingujung@hanyang.ac.kr}
\urladdr{https://sites.google.com/view/mingujung/}

\author[Raunig]{Tomáš Raunig}
\address[Raunig]{Institute of Mathematics, Czech Academy of Sciences, Žitná 25, 115 67, Czech Republic\newline second address: Faculty of Mathematics and Physics, Charles University, Sokolovská 83, Prague, 186 00 \newline
\href{https://orcid.org/0009-0001-3425-8726}{ORCID: \texttt{0009-0001-3425-8726}}}
\email{raunig@karlin.mff.cuni.cz}

\thanks{}


\begin{abstract} 
We introduce and study \emph{equivariant} versions of the Radon–Nikodým property for Banach spaces, together with the closely related notions such as dentability, the Bishop–Phelps and Krein–Milman properties, and Lindenstrauss’ property A, all considered in the presence of a continuous group action by linear isometries. While in the classical setting the Radon–Nikodým property, the Bishop–Phelps property and dentability are equivalent, the equivariant situation turns out to depend essentially on the acting group and requires non-trivial tools from abstract harmonic analysis and representation theory. We establish several implications among the equivariant counterparts of these properties. Namely, given a compact group $G$, the $G$-Bishop–Phelps property implies strong $G$-dentability, which in turn implies the $G$-Krein–Milman property and the classical Bishop-Phelps property, for any $G$-Banach space. Moreover, given a locally compact and second countable group $G$, the $G$-Radon-Nikodým property is equivalent to the classical Radon-Nikodým property, for any $G$-Banach space.
\end{abstract}

\maketitle

\tableofcontents

\section{Introduction}

The Radon-Nikod\'ym property for Banach spaces, which--as the name suggests--concerns Banach spaces for which the Radon-Nikod\'ym theorem for vector-valued measures holds, has been a central topic in Banach space theory for nearly a century. It first appeared in the works of Dunford \cite{Dun36,DunfordPettis1940} and played a pivotal role in the development of the theory of vector-valued measures. Although by definition an analytic property involving measures, the Radon-Nikod\'ym property (RNP, for short) admits numerous equivalent characterizations (see, for instance, \cite[Chapter VII.6]{DU1977}), some of them purely geometrical, that connect several major branches of Banach space theory. We refer the reader to the classical monographs \cite{JohnsonLindenstrauss2003, Bourgin1983, DU1977, Phelps1993}, as well as to more recent references such as \cite{FHHMZ2010, Pisier2016, Ryan2002}.

In this paper, we are interested in what can be dubbed an equivariant Radon-Nikod\'ym property; that is, the RNP in the presence of a continuous group action on a~Banach space by linear isometries. The motivations are at least twofold. First, and that was the original and foremost motivation for us, we want to extend the already rich and expanding research on norm-attaining operators that are invariant under a~group action by linear isometries. We refer to the papers \cite{DFJ2023, DFJR2023, Falco2021, FGJM2022} for some of the very recent developments in this direction, as well as to \cite{FI2024, FI2026} for related work. One of the cornerstones in the norm-attaining theory is the Bishop-Phelps theorem for functionals and its extension to more general operators turns out to be valid exactly for those Banach spaces having the RNP. The equivariant Bishop-Phelps theorem, that is, the version concerning group equivariant operators, is therefore a natural continuation in this line of work and a starting point of our study of the equivariant RNP.

Second, research in this direction belongs to the emerging area of equivariant Banach space theory. From the category-theoretical point of view, it amounts to upgrading the classical category of Banach spaces into an equivariant category--whose objects are Banach spaces endowed with group actions and whose morphisms are equivariant linear operators. This perspective parallels the developments that took place several decades ago in topology and geometry, where equivariant homotopy and co/homology theories have become central in the subject. There are now plenty of tools at our disposal. Indeed, equivariant Banach space theory can draw on representation theory of group actions on Banach spaces, which is particularly well developed for compact groups (see e.g. the monograph \cite{HofMorbook}), and on analytic group theory studying rigidity of group actions on Banach spaces (see e.g. \cite{BFGM2007,BHV2008}). For a general survey on category-theoretic and homology methods in Banach space theory, we refer to the modern monograph \cite{CS-Cbook} and to \cite{CaFe23} for a recent work in this subject in the equivariant setting.

Here we consider equivariant versions of the RNP and several related concepts, including the Riesz representation property, the already mentioned Bishop-Phelps property, dentability of bounded subsets, the Krein-Milman property and the Lindenstrauss' property A, all defined in the presence of a group action. While in the classical case, all these conditions, except for the last two which are implied by the rest, are equivalent, the equivariant setting turns out to be substantially more intricate. Several of the implications require non-trivial tools from abstract harmonic analysis and demand that the acting group be either compact, in which case one employs the Big Peter-Weyl theorem, or locally compact and $\sigma$-compact, taking advantage of the existence of a $\sigma$-finite Haar measure.

We review the main results below. Note that it is not an exhaustive list of the results but rather a summary of the main implications among the equivariant properties. Throughout the paper, we adopt the convention that equivariant (or sometimes invariant) properties are denoted by the prefix `$G$-’, where $G$ stands for a generic group and serves as a shorthand for \emph{$G$-equivariant} or \emph{$G$-invariant}.  For example, we write \emph{$G$-RNP} in place of the more accurate \emph{group equivariant Radon-Nikod\'ym property}. This notation will be used consistently for the reader’s convenience. For the precise definitions that appear in the following, we refer the reader to \Cref{section:definitions} below.

\noindent\textbf{Main results.}

\begin{enumerate}[topsep=0pt]
    \item The $G$-Bishop–Phelps property implies strong $G$-dentability, and thus also the classical Bishop-Phelps property, when $G$ is compact.
    \item Strong $G$-dentability implies the $G$-Krein–Milman property, again when $G$ is compact.
    \item Weak $G$-dentability implies the $G$-RNP, and the converse holds if $G$ is locally compact and second countable. In this case, the $G$-RNP is equivalent to the classical RNP.
\end{enumerate}
\

In addition, we present several examples showing that not all these implications can be reversed, together with further phenomena illustrating how the equivariant cases differ from the classical one. We also thoroughly investigate the equivariant analogue of Lindenstrauss' property A.

The paper is organized as follows. \Cref{section:prelim} recalls the Radon-Nikod\'ym property, the central notion of the paper, and sets the notation and basic notions for the sequel. \Cref{section:definitions} rigorously introduces all the new equivariant versions of the properties discussed informally above and shows their fundamental properties. The core of the paper, where, in particular, all the results stated above are proved, lies in \Cref{section:G-BP,section:G-dent,section:G-RNP}. \Cref{section:G-NA} deals with norm-attaining equivariant operators. Interestingly, the original motivation for this work has evolved into a natural offshoot of the main part of the paper, yet it naturally follows from the preceding sections and is of independent interest. Finally, \Cref{section:summary} summarizes the results and suggests open problems. A reader reaching this point will recognize that there is a plethora of new directions to be explored.

To conclude, we believe that developing an equivariant counterpart of the Radon-Nikod\'ym property opens a promising path toward a systematic understanding of Banach space geometry under group actions. In particular, the equivariant viewpoint not only extends the reach of classical Banach space theory but also provides a new language in which linear and group-theoretic aspects coexist naturally. The equivariant Radon–Nikod\'ym property, together with its related notions studied here, thus constitutes a natural starting point for further investigations at the intersection of functional analysis, abstract harmonic analysis, and geometric group theory.

\section{Preliminaries}\label{section:prelim}

As the main goal of this paper is the study of the group equivariant version of the Radon-Nikodým property for Banach spaces, we begin by briefly recalling its classical definition. For simplicity and as used already in the introduction above, we will be denoting this property as RNP. A Banach space $X$ is said to have the RNP if the classical Radon-Nikodým theorem (see, for instance, \cite[Chapter III]{DU1977}) extends to $X$-valued measures. More precisely, $X$ has the RNP if and only if for every $\sigma$-finite measure space 
$(\Omega, \Sigma, \lambda)$ and for every $X$-valued measure $\mu: \Sigma \to X$ 
that has bounded variation and is $\lambda$-absolutely continuous, 
there exists a $\lambda$-Bochner integrable function $g: \Omega \to X$ such that 
\begin{equation} \label{eq1}
    \mu(E) = \int_E g \, d\lambda
\end{equation}
for all $E \in \Sigma$. It is enough to consider the case where $\lambda(\Omega) < \infty$ and 
$\|\mu(E)\| \le \lambda(E)$ for all $E \in \Sigma$; in this situation, any derivative $g: \Omega \to X$ is necessarily bounded. 
Equivalently, $X$ satisfies the RNP if and only if for every vector measure 
$\mu: \Sigma \to X$ with $\|\mu(E)\| \le \lambda(E)$, 
there exists a bounded measurable function $g: \Omega \to X$ 
such that \eqref{eq1} holds for all $E \in \Sigma$. 
This provides an equivalent formulation of the RNP, which will be useful throughout the text. This is in turn equivalent to the statement that, for every continuous linear operator $T: L_1(\Omega,\Sigma,\lambda) \to X$, 
there exists a bounded measurable function $g: \Omega \to X$ such that 
\[
    T(f) = \int_\Omega f g \, d\lambda, \quad \text{for all } f \in L_1(\Omega,\Sigma,\lambda). 
\]

It is well known that $c_0$ and $L_1[0,1]$ do not satisfy the RNP. 
On the other hand, reflexive spaces do satisfy the RNP \cite{Phillips1943}, 
as do separable dual spaces \cite{DunfordPettis1940}. 
Moreover, if a Banach space $X$ has the RNP, then every closed subspace of $X$ also has it. 
Conversely, if every closed separable subspace of $X$ has the RNP \cite{Ronnow1967}, 
then $X$ itself has the RNP (see also \cite{Uhl1972}). 
In particular, if every separable subspace of $X$ is isomorphic to a subspace of a separable dual space, then $X$ has the RNP. In particular, $\ell_1(\Gamma)$ has the RNP for any uncountable set $\Gamma$. 

In what follows, we provide the basic notation and conventions we will be using throughout the paper. We follow standard notation coming, essentially, from the references \cite{DU1977, FHHMZ2010}.

\noindent
{\bf Notation}. In the present paper, all the Banach spaces $X$ will be considered over the \emph{real} numbers. We denote by $B_X$ its closed unit ball and by $S_X$ its unit sphere. We also denote by $X^*$ the topological dual of $X$. If $x \in X$ and $\e>0$, we denote by $B_\eps(x)$ the open ball of radius $\e$ given by $\{y \in X: \|x-y\| < \e\}$. Given a finite set $\{x_1, \ldots, x_n\} \subseteq X$, the symbol $B_{\e}(x_1, \ldots, x_n)$ stands for the union of open balls of radius $\e$ centered at the points $x_1, \ldots, x_n$, i.e., $B_{\e}(x_1, \ldots, x_n) = \bigcup_{i=1}^n B_{\e}(x_i)$. We denote by $\diam A$ the diameter of a subset $A$. The symbols $\co(A)$ and $\overline{\co}(A)$ stand for the convex hull and closed convex hull of $A$, respectively. We denote by $\mathcal{L}(X,Y)$ the Banach space of all bounded linear operators from $X$ into $Y$.

Let $G$ be a topological group, which we always assume to be Hausdorff. We denote by $1_G \in G$ the identity element of the group $G$. On a few occasions, we shall employ the Haar measure, which exists if and only if the group is locally compact, and which is the unique (up to a scalar multiple) non-trivial Radon Borel measure on the group which is invariant under left multiplication. The monograph \cite{HofMorbook} serves as our reference for topological groups, to which we refer for any unexplained notation. Given a Banach space $X$, an important example of a topological group is $\Iso(X)$, the group of all linear isometries equipped with the strong operator topology (SOT, for short), which is the topology of pointwise convergence in norm.

We shall denote the actions by the symbol $\alpha:G\curvearrowright X$ or simply by $G\curvearrowright X$, when there is no risk of confusion. For $g\in G$ and $x\in X$, we write $\alpha(g,x)\in X$ for the application of the element $g$ to $x$ under the action $\alpha$. As before when no confusion can arise, we omit the symbol $\alpha$ and simply write $g\cdot x$, or even $gx$, instead of $\alpha(g,x)$. For subsets $F\subseteq G$ and $A\subseteq X$ we write $F\cdot A$ to denote the set $\{f\cdot a\colon f\in F, a\in A\}$. All group actions considered in this paper are continuous. If $X$ is a Banach space, which will be the most common case, continuity of the action refers to the continuity of the map $G\times X\to X$, or equivalently, to the continuity of the corresponding homomorphism from $G$ into $\Iso(X)$. In case $X$ is a measure space, the continuity of the action is defined in the appropriate sense for that context.

Throughout the text, when convenient, we refer to a \emph{$G$-Banach space} as a Banach space endowed with a continuous action of a topological group $G$ by linear isometries.

Let $X$ and $Y$ be Banach spaces, and let $G$ be a topological group acting continuously by linear isometries on $X$ and $Y$. We say that a linear operator $T:X\to Y$ is \emph{$G$-equivariant}, or just \emph{equivariant}, if 
\[
T(g\cdot x)=g\cdot T(x), \,\, \text{for all $g\in G$ and $x\in X$}.
\]
We denote by $\mathcal{L}^G(X,Y)$ the Banach space of all $G$-equivariant bounded linear operators from $X$ into $Y$. In the language of representation theory, such operators are often referred to as \emph{intertwining operators}. Finally, a subset $C$ of a $G$-Banach space is said to be \emph{$G$-invariant} if $g\cdot C \subseteq C$ for every $g \in G$.

\section{Definitions, basic properties and examples} \label{section:definitions}

In this section, we introduce the central definitions of the paper and show their basic properties. More specifically, we focus on the $G$-versions of dentability, Radon-Nikodým, Bishop-Phelps, Krein-Milman, and Lindenstrauss' property A.

\subsection{$G$-Bishop-Phelps property}

Recall that a nonempty bounded closed subset $C$ of a Banach space $X$ is said to have the {\it Bishop-Phelps property} (BP, for short) whenever given any Banach space $Y$, any bounded linear operator $T \in \mathcal{L}(X,Y)$ and any $\delta>0$, there exists $S \in \mathcal{L}(X,Y)$ such that $\|S - T\| < \delta$ and $\sup \{ \|S(x)\|: x \in C \}$ is attained. We will say then that the Banach space $X$ has the BP if every bounded closed absolutely convex subset of $X$ has the BP.

Bourgain's result (see \cite[Proposition 1]{Bourgain1977}) states that if $X$ has the Bishop-Phelps property then it is dentable (recall its definition at the beginning of \Cref{Subsection-dentability} below). One of our goals here is to obtain an equivariant version of this result (see \Cref{thm:G-BP_implies_G-dentable} below) by putting some conditions on the group $G$ such that an analogous conclusion holds true for $G$-equivariant operators. We consider then the following natural definition.

\begin{definition}[$G$-Bishop-Phelps property] \label{definition:G-BP} Let $G$ be a topological group, and let $C$ be a bounded closed absolutely convex $G$-invariant subset of a $G$-Banach space $X$. We say that $C$ has the {\it $G$-Bishop-Phelps property} ($G$-BP, for short) whenever, for every $G$-Banach space $Y$, for any $G$-equivariant operator $T \in \mathcal{L}^G(X,Y)$ and $\delta > 0$, there exists another $G$-equivariant operator $T_{\delta} \in \mathcal{L}^G(X,Y)$ such that $\|T_{\delta} - T\| < \delta$ and $\sup \{ \|T_{\delta}(x)\|: x \in C \}$ is attained. We will say that $X$ has the $G$-BP if every bounded closed absolutely convex $G$-invariant subset of $X$ has the $G$-BP.
\end{definition}

A natural choice is to consider $C = B_X$ for which we obtain the following specialized definition--one of the original motivations for our research. Before introducing this new definition, let us provide some notation. Let $X$ and $Y$ be $G$-Banach spaces. We denote by $\NA^G(X, Y)$ the set of all $G$-equivariant norm-attaining operators $T \in \mathcal{L}^G(X, Y)$, that is, those for which there exists $x \in B_X$ satisfying $\|Tx\| = \|T\|$.

\begin{definition}\label{def:G-propertyA} Let $G$ be a topological group. A $G$-Banach space $X$ is said to have \emph{$G$-property A} (in the sense of Lindenstrauss) if for every $G$-Banach space $Y$, the set $\NA^G(X,Y)$ is dense in $\mathcal{L}^G(X,Y)$.
\end{definition}

\hypertarget{GBPnotimpliesBP}{
Let us remark that it is not at all clear whether the classical Bishop-Phelps property implies the $G$-Bishop-Phelps property. However, we do know that the converse is not true since $L_1$ has the $G$-BP, for $G=\Iso(L_1)$, by \Cref{cor:ATimpliesG-BP} below, while $L_1$ does not have the BP as $L_1$ is clearly not dentable.
}

\subsection{$G$-dentability and weaker notions} \label{Subsection-dentability} In this subsection, we introduce our $G$-versions of dentability. Recall that a subset $C \subseteq X$ of a Banach space $X$ is said to be dentable if for every $\varepsilon > 0$, there exists a point $x \in C$ such that $x \not\in \overline{\co}(C \setminus B(x, \varepsilon))$. The Banach space $X$ is said to be dentable if every nonempty bounded subset of $X$ is dentable.

If $X$ is a $G$-Banach space for some topological group $G$, then for every $C \subseteq X$, we denote by $\clco_G (C)$ the smallest closed convex $G$-invariant subset of $X$ containing the subset $C$.

We introduce the following $G$-version of dentability.

\begin{definition}\label{def:G-dent_new} Let $G$ be a topological group. 
    Given a $G$-invariant subset $C$ of a $G$-Banach space $X$, we say that $C \subseteq X$ is \textit{$G$-dentable} if for every $\varepsilon > 0$, there exists a finite set $\{x_1, \ldots, x_n\} \subseteq C$ such that one can find an element $x \in C$ such that
\begin{equation*} 
x \notin \overline{\co}_G\big(C \setminus B_\varepsilon(x_1, \ldots, x_n)\big).
\end{equation*}
\end{definition}

Recall from Huff and Morris \cite{HuffMorris1976} that a subset $C$ of $X$ is dentable if and only if it is $\{1_G\}$-dentable in the sense of \Cref{def:G-dent_new}. This geometric characterization of dentability, due to Huff and Morris, is in fact the main motivation for our definition of $G$-dentability for subsets. 
Observe that for any topological group $G$ and $G$-invariant subset $C \subseteq X$, $G$-dentability of $C$ implies its classical dentability.

Let us collect two more equivalent but formally weaker definitions of $G$-dentability which will be more convenient to verify in later proofs. The following lemma is a purely notational equivalence and for this reason we omit its proof.

\begin{lemma} \label{fact:equivDefOfStrongGDentability} Let $G$ be a topological group and let $X$ be a $G$-Banach space. Let $C \subseteq X$ be a $G$-invariant subset. The following conditions are equivalent.
	\begin{enumerate}
		\item for every $\eps > 0$, there is a finite set $\{x_1,\ldots, x_n\} \subseteq C$ such that there exists $x \in C$ with $x \notin \clco_G ( C \setminus B_\eps (x_1,\ldots, x_n))$.
		\item for every $\eps > 0$, there is a finite set $\{x_1,\ldots, x_n\} \subseteq X$ such that there exists $x \in C$ with $x \notin \clco_G ( C \setminus B_\eps (x_1,\ldots, x_n))$.
		\item for every $\eps > 0$, there exist a finite family of subsets $A_1, \dots, A_n \subseteq X$ with $\diam A_i < \eps$ for every $i$ and $x \in C$ such that $x \notin \clco_G ( C \setminus \bigcup_{i=1}^n A_i)$.
	\end{enumerate}
\end{lemma}

Let us now present natural group-invariant analogues of the notion of dentability for Banach spaces.

\begin{definition}[$G$-dentabilities for Banach spaces]\label{def:dentable_banach_spaces}
Let $G$ be a topological group and let $X$ be a $G$-Banach space. We say that $X$ is
    \begin{enumerate}
        \itemsep0.25em
        \item \textit{strongly $G$-dentable} if every bounded $G$-invariant subset $C$ of $X$ is $G$-dentable in the sense of \Cref{def:G-dent_new}.
        
        \item \textit{weakly $G$-dentable} if every bounded $G$-invariant subset $C$ of $X$ is dentable in the classical sense.      
    \end{enumerate}
\end{definition}

\begin{remark} \label{remark:G-dentability}   We make the following observations concerning \Cref{def:dentable_banach_spaces}. See also Diagram~\ref{diagram-dentability1} below.
\begin{itemize} 
\itemsep0.25em
\item[(a)] Strong $G$-dentability immediately implies weak $G$-dentability.

\item[(b)] Weak $G$-dentability is in fact equivalent to classical dentability for $G$-Banach spaces. See \Cref{prop:weakly-G-dentable-and-dentable-are-equivalent} and the preceding discussion.

\item[(c)] Strong $G$-dentability is a stronger condition that forces the acting group to be compact, or more precisely, precompact. See \Cref{fact:strongG-dent}.

\end{itemize}
\end{remark}

\begin{remark}
One could also introduce an even weaker notion of $G$-dentability for a $G$-Banach space $X$ by requiring that every bounded $G$-invariant subset $C$ of $X$ satisfy the following condition: for every $\varepsilon>0$, there is $x \in C$ such that 
\[
x \not\in \clco (C\setminus G\cdot B_\varepsilon (x)).
\]
This notion, however, turns out to be too weak. Indeed, every $G$-Banach space $X$ on which $G$ acts almost transitively is $G$-dentable in this sense. In particular, $L_1$ is $G$-dentable with respect to the action of $G=\Iso(L_1)$.

To see this, let $C \subseteq X$ be bounded and $G$-invariant and let $\eps > 0$ be arbitrary. Put $M = \sup \{ \|y\| : y \in C\}$. If $M=0$, then $C = \{0\}$ and the conclusion is immediate by taking $x=0$. Assume therefore that $M > 0$. Choose $0 < \delta < \min \{ \e, M \}$
    and find $x \in C$ such that $\|x\| > M - \delta$. We claim that 
    \[
    \clco (C \setminus G \cdot B_\eps(x)) \subseteq \overline{B_{M - \delta}(0)}
    \]
    which immediately implies $x \notin \clco (C \setminus G \cdot B_\eps(x))$.
    
    Indeed, let $y \in C$ be such that $\|y\| > M - \delta$.  Then both $x$ and $y$ are nonzero. Moreover, $|\|x\|-\|y\|| < \delta < \e$. Set $\eta := \eps - \big|\|x\| - \|y\|\big| >0$. Since the action of $G$ on $X$ is almost transitive, there is $g \in G$ such that
    \begin{equation*}
        \left\| g \cdot x - \frac{\|x\|}{\|y\|}y \right\| < \eta.
    \end{equation*}
    Thus, we also have
    \begin{equation*}
        \| g\cdot x - y \|
        \leq \left\| g\cdot x - \frac{\|x\|}{\|y\|}y \right\| + \left\| \frac{\|x\|}{\|y\|}y - y \right\|
        < \eta + \big|\|x\| - \|y\|\big| = \eps.
    \end{equation*}
    Thus, $y \in G\cdot B_\varepsilon (x)$. Consequently, $C \setminus G \cdot B_\eps(x) \subseteq B_{M-\delta } (0)$ and the same inclusion holds for its closed convex hull, completing the proof of the claim. 
\end{remark}

Although simple, the following proposition plays a crucial role in relating several equivariant notions to their classical counterparts. It has a number of important consequences, such as \Cref{cor:G-BPimpliesBP} and \Cref{cor:G-RNPisRNP}, which we do not know how to prove directly. Instead, their proofs proceed through a sequence of intermediate implications in which the proposition below serves as the essential bridge between the classical and equivariant settings. For this reason, despite its equivalence with classical dentability, we shall continue to use the term ``weak $G$-dentability'' throughout the remainder of the paper.

\begin{proposition} \label{prop:weakly-G-dentable-and-dentable-are-equivalent}  Let $G$ be any topological group and $X$ a $G$-Banach space. Let $D \subseteq X$ be a nonempty bounded subset. If $D$ is not dentable, then $G \cdot D$ is not dentable. In particular, $X$ is weakly $G$-dentable if and only if $X$ is dentable.
\end{proposition}

\begin{proof} Since $D$ is not dentable, there exists $\e_0 > 0$ such that $x \in \overline{\co}(D \setminus B(x, \e_0))$ for every $x \in D$. We will show that the same $\varepsilon_0$ witnesses the non-dentability of $G\cdot D$. First, notice that $G \cdot D$ is bounded. Let $y \in G \cdot D$ and write $y= g\cdot x$ for some $g \in G$ and $x \in D$. Then, $x \in \overline{\co}(D \setminus B_{\e_0} (x))$. Applying $g$, we obtain that 
\begin{equation*}
    y = g \cdot x \in g \cdot \overline{\co}(D \setminus B_{\e_0} (x)) = \overline{\co} (g \cdot (D \setminus B_{\e_0} (x)).
\end{equation*}
Since  
\[
g \cdot (D \setminus B_{\e_0} (x)) = (g \cdot D )\setminus B_{\e_0}(g \cdot x) = (g \cdot D) \setminus B_{\e_0}(y),
\]
it follows that $y \in \overline{\co}( (g \cdot D )\setminus B_{\e_0}(y)) \subseteq \clco ( (G \cdot D) \setminus B_{\e_0}(y) ) $. Since $y\in G\cdot D$ was arbitrary, $G \cdot D$ is not dentable.

To prove the ``In particular'' statement, it suffices to show that weak $G$-dentability implies classical dentability. Let $D\subseteq X$ be a nonempty bounded subset. We have just shown that if $D$ is not dentable, then neither is $G\cdot D$. Since $G\cdot D$ is a bounded $G$-invariant subset of $X$, it follows that if $X$ is weakly $G$-dentable, then $G\cdot D$ is dentable. Therefore, $D$ must be dentable.
\end{proof}

Now we explain \Cref{remark:G-dentability} (c). First, recall that a topological group $G$ is \textit{precompact} if for every neighborhood $U$ of $1_G \in G$, there exists a finite set $F_U$ such that 
 \[
 G \subseteq \bigcup_{g \in F_U} g\cdot  U.
 \]
Equivalently, the completion by its left uniformity is compact.

The assumption in \Cref{fact:strongG-dent} below that the action is \emph{faithful}, i.e., that for every $g\in G\setminus \{1_G\}$, there exists $x\in X$ such that $g\cdot x\neq x$, is not restrictive. In fact, if the action is not faithful, then the set $N:=\{g\in G\colon g\cdot x = x, \, \forall x\in X \}$ is a non-trivial closed normal subgroup of $G$. In this case, the action of $G$ on $X$ factorizes through the natural action of $G/N$ on $X$ given by 
 \[
 gN \cdot x = g\cdot x \quad \text{($g \in G$, $x \in X$)},
 \]
 which is faithful. 

\begin{fact}\label{fact:strongG-dent}
Let $X$ be a Banach space equipped with a faithful action of a topological group $G$ by linear isometries. Suppose that $X$ is strongly $G$-dentable. Then, there exists a coarser group Hausdorff topology $\tau$ on $G$ such that $(G,\tau)$ is precompact and such that the action is still continuous.    
\end{fact}
 
\begin{proof} 
 Since the action is faithful, we can identify $G$ with $\phi[G]$, where $\phi$ is the natural continuous homomorphism $\phi:G\to\Iso(X)$. Identifying $G$ with $\phi[G]$ naturally induces a topology $\tau_\phi$ on $G$, where a net $g_\alpha \to g$ in $(G,\tau_\phi)$ if and only if $\phi(g_\alpha) \to \phi(g)$ in $(\Iso(X), \text{SOT})$, which is coarser than the originally given topology on $G$ (by continuity of $\phi$). 
 We claim that $\phi [G]$ is precompact in $\Iso (X)$.
 
 First notice that every orbit of $G$ is totally bounded. Pick any $x\in X$. Since by the assumption the orbit $Gx$ is $G$-dentable, it follows that, for every $\varepsilon>0$, there are $x_1,\ldots,x_n\in Gx$ such that 
 \[
 Gx\not\subseteq \clco_G  ( Gx\setminus B_\varepsilon(x_1,\ldots,x_n) ).
 \] However, for every nonempty subset $C\subseteq Gx$, we have $Gx\subseteq\clco_G C$, which implies that $Gx\setminus B_\varepsilon(x_1,\ldots,x_n)$ must be empty. Thus, the orbit $Gx$ has a finite $\varepsilon$-net. Since $\varepsilon>0$ was arbitrary, it follows that $Gx$ is totally bounded; equivalently, precompact (as the ambient space $X$ is a complete metric space).

Now we prove the claim. Let $U$ be an arbitrary open neighborhood of $\id_X$. Without loss of generality, we may assume that $U=\{f\in \Iso(X)\colon \|f(x_i)-x_i\|<\varepsilon, \forall i\leq n\}$ for some $n \in \N$, $\varepsilon>0$ and $(x_i)_{i=1}^n\subseteq X$. 

For each $i\leq n$, the orbit $G\cdot x_i$ is totally bounded, so there exists a finite set $F_i\subseteq G$ such that for all $g\in G$ there is $f\in F_i$ such that $\|fx_i-gx_i\|<\varepsilon/2$. For each $(f_i)_{i=1}^n\in \prod_{i=1}^n F_i$ define 
\[
\mathcal{F}_{(f_i)_i}:=\{g\in G\colon \|gx_i-f_i x_i\|<\varepsilon/2 \;\;\forall i\leq n\}.\] 
If $\mathcal{F}_{(f_i)_i}\neq\emptyset$, then choose $g_{(f_i)_i}\in \mathcal{F}_{(f_i)_i}$, otherwise let $g_{(f_i)_i}=1_G$. Set 
\[
F:= \left\{ g_{(f_i)_i} : (f_i)_i\in \prod_{i=1}^n F_i \right\},
\]
which is finite. We show that $G\subseteq F\cdot U$. Indeed, let $g\in G$. By construction, there exists $(f_i)_i\in\prod_{i=1}^n F_i$ such that $g\in\mathcal{F}_{(f_i)_i}$. Hence, for every $i\leq n$
\[\|gx_i-g_{(f_i)_i} x_i\|\leq \|gx_i- f_i x_i\|+\|f_i x_i-g_{(f_i)_i} x_i\|<\varepsilon/2+\varepsilon/2=\varepsilon.
\] Therefore, $g_{(f_i)_i}^{-1}g\in U$; hence 
\[
g \in g_{(f_i)_i} U \subseteq F\cdot U.
\]
Since $g\in G$ was arbitrary, we conclude that $G\subseteq F\cdot U$, as required.
\end{proof}

The argument of \Cref{fact:strongG-dent} can be used to distinguish between weak and strong $G$-dentabilities, as follows.

\begin{corollary} \label{cor:strongGdentIsStrong}
Weak $G$-dentability, and therefore classical dentability, does not imply strong $G$-dentability. For instance, for $1<p<\infty$, $L_p$ is not strongly $G$-dentable with respect to $G=\Iso(L_p)$. 
\end{corollary} \label{counterexample-G-dentable}

\begin{proof}
If $L_p$ were strongly $G$-dentable with respect to $G=\Iso(L_p)$, then the proof of \Cref{fact:strongG-dent} would imply that $(\Iso(L_p), \text{SOT})$ is compact, which is not the case--a fact well known to experts, and easily deducible from, for instance, \cite[Proposition 3.2]{FerencziRosendal2013}. On the other hand, $L_p$ is dentable, thus it is also weakly $G$-dentable.
\end{proof}

In \Cref{section:G-BP}, we will also need the following auxiliary result, which will be used in the proof that the $G$-Bishop–Phelps property implies strong $G$-dentability.

\begin{lemma} \label{fact:sGdentabilityInheritance} Let $G$ be a topological group, $X$ be a $G$-Banach space, and let $D$ be a $G$-invariant subset of $X$. If $\clco(D)$ is $G$-dentable, then so is $D$.
\end{lemma}

\begin{proof}
	Let $\eps > 0$ and find $z \in \clco (D)$ and $y_1, \dots y_n \in \clco (D)$ such that
	\begin{equation*}
		z \notin \clco_G \big( (\clco (D)) \setminus B_\eps(y_1, \dots, y_n) \big) =: C.
	\end{equation*}
	Let $U$ be an open neighborhood of $z$ such that $U \cap C = \emptyset$. Since $z \in \clco (D)$, $U \cap \co (D) \neq \emptyset$. Hence, there exists 
    \[
    x = \sum_{j=1}^m \lambda_j x_j \in \co (D)  \setminus C
    \]
    where $x_j \in D$ for $1\leq j \leq m$, for some $m \in \mathbb{N}$. For a~contradiction, assume that
	\begin{equation*}
		\forall 1 \leq j \leq m : x_j \in \clco_G \big( D \setminus B_\eps(y_1, \dots, y_n) \big).
	\end{equation*}
	Enlarging $D$ to $\clco (D)$, we get
	\begin{equation*}
		\forall 1 \leq j \leq m : x_j \in \clco_G \big( \clco (D) \setminus B_\eps(y_1, \dots, y_n) \big) = C.
	\end{equation*}
	But since $C$ is convex, we obtain that $x = \sum_{j=1}^m \lambda_j x_j \in C$ which is a~contradiction with the very choice of $x$. This means that for some $j \in \{1, \dots, m\}$, we have that
	\begin{equation*}
		x_j \notin \clco_G \big( D \setminus B_\eps(y_1, \dots, y_n) \big).
	\end{equation*}
It follows from \Cref{fact:equivDefOfStrongGDentability} (since $y_1,\ldots,y_n$ are not necessarily elements of $D$) that $D$ is $G$-dentable. 
\end{proof}

\begin{figure}[h] 
\centering
\begin{tikzpicture}[
		arrow/.style={-Implies, double equal sign distance, thick},
		equiv/.style={Implies-Implies, double equal sign distance, thick}, 
		box/.style={draw, rounded corners, minimum width=3.6cm, minimum height=1.2cm, align=center, font=\normalsize}
		]
		
		\matrix (m) [matrix of nodes, row sep=1.2cm, column sep=1.6cm, nodes={anchor=center}] {
			& \node[box] (strong) {Strong\\$G$-dentability}; \\
			\node[box] (classical) {Dentability}; & \node[box] (weak) {Weak\\$G$-dentability}; \\
		};
		
		\draw[arrow] (strong) -- (weak);       
		
		\draw[equiv] (classical) -- (weak);    
		\draw[arrow] 
		    ([yshift=0.12cm]weak.east)
		    to[out=20, in=-10, looseness=1.2]
		    node[midway,sloped,xshift=0.1cm]{\LARGE $\times$}
		    ([yshift=0.15cm]strong.east);

        \draw[arrow]
    (classical.north)
    to[out=85,in=180,looseness=0.9]
    node[midway]{\rotatebox{12}{\LARGE $\times$}}
    (strong.west);
		
	\end{tikzpicture}
\captionsetup{labelformat=simple, labelsep=colon, name=Diagram}
\caption{$G$-dentability concepts and their relations.}
\label{diagram-dentability1}
\end{figure}

\subsection{$G$-Krein-Milman property} In this subsection, we consider a version of the Krein-Milman property that is compatible with a given group action. In \Cref{subsection:strongly-G-dentable-G-KMP}, we will show its connection with the $G$-dentability of a Banach space $X$ (more specifically, see Theorem \ref{theorem:strong-G-dentable-implies-G-KMP}).

Recall that if $C$ is a nonempty convex subset of $X$, a point $x_0 \in C$ is an \emph{extreme point} of $C$ if $x_1 = x_2 = x_0$ whenever $x_1, x_2 \in C$ and $x_0$ can be written as $x_0 = (1/2)x_1 + (1/2)x_2$. The well-known Krein-Milman theorem says that if $C$ is a nonempty convex compact subset of $X$, then $C$ contains at least one extreme point and, furthermore, $C = \overline{\co}(\ext(C))$ (see, for instance, \cite[Theorem 3.65]{FHHMZ2010}).

We say that a Banach space $X$ has the {\it Krein-Milman property} (KMP, for short) if every closed bounded convex subset of $X$ is the norm closed convex hull of its extreme points. The fact that every Banach space with the RNP has the KMP is due to J. Lindenstrauss (see, for instance, \cite[Theorem 7, p. 191]{DU1977}). 

The version we introduce below, the so-called $G$-KMP, is the natural group-compatible analogue of the KMP.

\begin{definition} \label{def:G-KMP}
    Let $G$ be a topological group, $X$ be a $G$-Banach space, and let $C \subseteq X$ be a closed bounded convex $G$-invariant subset of $X$. We say that $C$ has the \emph{$G$-Krein-Milman property} ($G$-KMP, for short) whenever $C = \overline{\co}_G(\ext(C))$.	We say that the Banach space $X$ satisfies the \emph{$G$-KMP} if every closed bounded convex $G$-invariant subset of $X$ has the $G$-KMP. 
\end{definition}

Let us notice that $\overline{\co}_G(\ext(C)) = \overline{\co}(\ext(C))$ since the set of extreme points in a $G$-invariant set is $G$-invariant. This means we have in fact that $C$ satisfies the $G$-KMP whenever $C = \overline{\co} (\ext(C))$.

\subsection{$G$-Radon-Nikod\'ym property} In this subsection, we introduce the group equivariant version of the Radon–Nikodým property for Banach spaces, the main topic of the present paper. As we know from the classical scenario, the dentability of a Banach space is equivalent to the RNP (see, for instance, \cite[Theorem 2.9]{Pisier2016}). Our aim is to prove that, under some hypothesis on the group $G$, weak $G$-dentability is equivalent to the $G$-RNP (see \Cref{def:newG-RNP} and \Cref{thm:G-RNPimpliesG-dentability}).     

Before providing the definition of the $G$-Radon-Nikod\'ym property, let us recall the relations between linear isometries of $L_p$ spaces for $p\in\left[1,\infty\right\}\setminus\{2\}$ and isomorphisms of measure spaces.

\begin{definition}[{\cite[Definition 3.2.3]{FlemingJamison2002}}]
Let $(\Omega_1, \Sigma_1, \mu_1)$ and $(\Omega_2, \Sigma_2, \mu_2)$ be measure spaces. A set map $T: \Sigma_1 \to \Sigma_2$ defined modulo null sets (i.e., $T(A)=T(B)$, for $A,B\in\Sigma_1$, if $\mu_1(A\triangle B)=0$) is called a \emph{regular set isomorphism} if 
    \begin{enumerate}
        \itemsep0.25em 
        \item $T(\Omega_1 \setminus A) = T\Omega_1 \setminus TA$ for all $A \in \Sigma_1$.
        \item $T(\cup_{n=1}^\infty A_n)=\cup_{n=1}^\infty TA_n$ for disjoint $A_n \in \Sigma_1$.
        \item $\mu_2 (TA)=0$ if and only if $\mu_1 (A)=0$.
    \end{enumerate}
All set equalities and inclusions are understood to be modulo null sets. 

Denoting by \[N_i:=\{A\in\Sigma_i\colon \mu_i(A)=0\},\quad i\in\{1,2\}\] the corresponding null ideals, a regular set isomorphism induces an isomorphism of the quotient measure algebras $\Sigma_1/N_1$ and $\Sigma_2/N_2$, and conversely, every such an isomorphism induces a regular set isomorphism.

\end{definition}

\begin{theorem}[Banach-Lamperti (see, for instance, {\cite[Theorem 3.2.5]{FlemingJamison2002}})]\label{thm:BanachLamperti}
    Suppose that $U: L_p(\Omega_1, \Sigma_1, \mu_1) \rightarrow L_p(\Omega_2, \Sigma_2, \mu_2)$ is a surjective linear isometry, where $1 \leq p < \infty$, $p\not=2$, and the measure spaces are $\sigma$-finite. Then, there exists a regular set isomorphism $T:\Sigma_1\to\Sigma_2$ and a function $h$ defined on $\Omega_2$ such that 
    \begin{equation}\label{eq:BL}
    U(f)(t) = h(t) (\widetilde{T}f)(t)
    \end{equation}
    where $\widetilde{T} :L_0 (\Omega_1,\Sigma_1, \mu_1) \to L_0 (\Omega_2, \Sigma_2, \mu_2)$ is a linear operator induced from $T$ defined by 
    \begin{equation}\label{def:T_tilde}
    \widetilde{T}f(t)=s\quad\text{if}\quad t\in \Big(\bigcap_{r\in\mathbb{Q}\setminus (-\infty,s]}T\big[f^{-1}\left(-\infty,r\right]\big]\setminus \bigcup_{r\in\mathbb{Q}\setminus [s,\infty)}T\big[f^{-1}\left(-\infty,r\right]\big]\Big),    
    \end{equation}
     and where $L_0$ denotes the space of real-valued measurable functions on the corresponding measure space (see \cite[Remarks 3.2.4]{FlemingJamison2002} for more details about $\widetilde{T}$). In particular, $\widetilde{T}(\chi_A)= \chi_{TA}$ for all $A \in\Sigma_1$. Moreover, $h$ satisfies that 
    \[
    \int_{TA} |h|^p \, d\mu_2 = \int_{TA} \frac{d( \mu_1 \circ T^{-1} )}{d\mu_2} \, d\mu_2 = \mu_1(A), \quad \forall A \in \Sigma_1.
    \]
    Conversely, for any $h$ and $T$ as above, the operator $U$ satisfying \eqref{eq:BL} is an isometry.

    Moreover, the linear isometry $U$ is positive if and only if the function $h$ is positive, i.e., equal to $\big(\frac{d(\mu_1\circ T^{-1})}{d\mu_2}\big)^{1/p}$.
\end{theorem}

Let $(\Omega,\Sigma,\mu)$ be a $\sigma$-finite measure space. The set of all regular set isomorphisms on $(\Omega,\Sigma,\mu)$ forms a group, denoted by $\Aut^*(\Omega,\Sigma,\mu)$, or just $\Aut^*(\mu)$ if there is no risk of confusion. To each $T \in \Aut^* (\mu)$, we can assign the isometry $U_T \in \Iso (L_1 (\Omega,\Sigma,\mu))$ given by 
\[
U_T (f) (t) = \frac{d(\mu\circ T^{-1})}{d\mu} (t) (\widetilde{T} f)(t), 
\]
where {$\widetilde{T}:L_0(\Omega,\Sigma,\mu)\to L_0(\Omega,\Sigma,\mu)$ is defined as in \eqref{def:T_tilde}}. The map $\Phi : T \mapsto U_T$ is known to be an algebraic isomorphism between $\Aut^* (\mu)$ and the group $\Iso^+ (L_1(\Omega,\Sigma,\mu))$ of positive linear isometries on $L_1(\Omega,\Sigma,\mu)$ (see, e.g., \cite[Theorem 3.1]{Lamperti1958}). Note that $\frac{d(\mu\circ T^{-1})}{d\mu} (\cdot)$ is a nonnegative measurable function.

We equip $\Aut^*(\mu)$ with the topology obtained from the isomorphism $\Phi$. Moreover, given a topological group $G$, notice that actions of $G$ on $(\Omega,\Sigma,\mu)$ by regular set isomorphisms are in one-to-one correspondence with homomorphisms from $G$ into $\Aut^*(\mu)$. We say that an action of $G$ on $(\Omega,\Sigma,\mu)$ by regular set isomorphisms is \emph{continuous} if the corresponding homomorphism from $G$ into $\Aut^*(\mu)$ is continuous.

The Banach–Lamperti theorem (\Cref{thm:BanachLamperti}) then immediately yields the following result.

\begin{corollary}\label{cor:BanachLamperti}
Given a topological group $G$ and a finite measure space $(\Omega,\Sigma,\mu)$, continuous actions of $G$ on $(\Omega,\Sigma,\mu)$ by regular set isomorphisms are in one-to-one correspondence with continuous actions of $G$ on $L_1(\Omega,\Sigma,\mu)$ by positive linear isometries.
\end{corollary}

\subsubsection{$G$-quasi-invariance of measures} \label{subsubsec:GquasiInv}

The first step towards defining the $G$-Radon-Nikod\'ym property is to define what it means for a measure to be $G$-quasi-invariant. The informal idea is the following: we look for a condition that would be equivalent to the linear operator defined by $m$ being $G$-equivariant. It turns out that given a finite measure space $(\Omega, \Sigma, \mu)$ equipped with a continuous action of $G$ by regular set isomorphisms and a vector measure $m: \Sigma \to X$, we need, for every $g \in G$ and $A \in \Sigma$, that
\begin{equation*}
    g \cdot m(A) = \int_{gA} \frac{d g_*\mu}{d\mu} \ \d m,
\end{equation*}
which generalizes the statement that $g\cdot m(A)=m(gA)$ when $\mu$ is a $G$-invariant measure. (For readers who are not comfortable with vector-valued integration, we note that the formal definition given below does not involve it.)

The slightly more delicate issue is to identify the appropriate assumptions under which this idea can be rigorously formalized. This subsection is devoted to discussing possible options of these assumptions and their relations. Since $G$-quasi-invariance constitutes the central hypothesis in the $G$-RNP, any assumptions built into its formulation inevitably become part of the $G$-RNP framework. For this reason, it is worthwhile to examine these alternatives in some detail.

Let $X$ be a Banach space and $(\Omega,\Sigma)$ be a measurable space. Consider an $X$-valued vector measure $m:\Sigma\to X$. If $m$ has finite total variation, then the formula 
\begin{equation}\label{def:T_m}
    T_m(\chi_A) = m(A), \quad A \in \Sigma
\end{equation}
defines a bounded linear operator $T_m: L_1(\Omega, \Sigma, |m|) \rightarrow X$. Indeed, for a simple function $s = \sum_{i=1}^n \lambda_i \chi_{A_i}$, where $A_1, \ldots, A_n$ are pairwise disjoint measurable sets, define $T_m(s) = \sum_{i=1}^n \lambda_i m(A_i)$ and we have 
\begin{equation*}
    \|T_m(s)\| \leq \sum_{i=1}^n |\lambda_i| \|m(A_i)\| \leq \sum_{i=1}^n |\lambda_i| |m|(A_i) = \|s\|_{L_1(|m|)}.
\end{equation*}
This inequality shows that $T_m$ extends uniquely to a bounded linear operator $T_m:L_1(\Omega, \Sigma, |m|) \rightarrow X$ with $\|T_m\| \leq 1$ (see, for instance, \cite[Proposition 2.2]{Pisier2016}).

Throughout this subsection, we shall repeatedly pass between a vector measure $m$ and its associated operator $T_m$, and we will freely use the notation $T_m$ for the latter. The operator $T_m$ should be viewed as the operator-theoretic counterpart of $m$.

We will now give the definition of $G$-quasi-invariance we chose as the main definition and then compare it with alternatives.

\begin{definition}\label{def:quasi-invariant-vector-measure}
    Let $G$ be a topological group and $X$ be a $G$-Banach space. Suppose that $G$ acts continuously on a finite measure space $(\Omega,\Sigma,\mu)$ by regular set isomorphisms. Let $m: \Sigma \to X$ be an $X$-valued measure on $\Sigma$ with finite total variation such that $|m| \ll \mu$. We say that $m$ is \emph{$G$-quasi-invariant} if the following two conditions hold:
        \begin{enumerate}[label=(\arabic*)]
        \item \label{common-condition-quasi-invariant_first_condition}
            \begin{equation*}
                \forall g \in G : \frac{dg_{*}\mu}{d\mu} \in L_1(|m|),
            \end{equation*}
            where $g_*\mu$ is the pushforward measure defined by $g_*\mu(A):=\mu(g^{-1}A)$ and $\frac{dg_*\mu}{d\mu}$ is the Radon-Nikod\'ym derivative of $g_*\mu$ with respect to $\mu$, and
        \item \label{common-condition-quasi-invariant_second_condition}
            \begin{equation*} 
                \forall g \in G \ \forall A \in \Sigma : g \cdot m(A) = T_m \left( \frac{dg_{*} \mu}{d\mu} \chi_{gA} \right).
            \end{equation*}
    \end{enumerate}
\end{definition}

\begin{remark}\label{rem:G-quasi-invariance}
    The definition deserves a few remarks:
    \begin{enumerate}[label=(\alph*)] \label{rem:defQuasiInvVectorMeasure}

        \item Condition \ref{common-condition-quasi-invariant_first_condition} serves to ensure that the condition \ref{common-condition-quasi-invariant_second_condition} is well-defined: indeed, it guarantees that the expression $T_m \left( \frac{dg_{*} \mu}{d\mu} \chi_{gA} \right)$ is well-defined for all $g \in G$ and $A \in \Sigma$.

        \item If the action of $G$ on $(\Omega,\Sigma,\mu)$ preserves the measure, i.e., for every $g\in G$ and $A\in\Sigma$ we have $\mu(gA)=\mu(A)$, then an $X$-valued finite-variation measure $m$ is $G$-quasi-invariant if and only if $m$ is $G$-equivariant; that is, for all $g\in G$ and $A\in\Sigma$, $m(gA)=g\cdot  m(A)$.

        \item\label{rem:QuasiInvVecMeasureInNiceCase} If one is lucky enough to work in a setting where $\abs{m} \leq C\mu$ for some $C > 0$, then condition \ref{common-condition-quasi-invariant_first_condition} is automatic. Indeed, this follows from the fact that $\frac{dg_*\mu}{d\mu}\in L_1(\mu)$, for all $g\in G$, and that $L_1(\mu)\subseteq L_1(|m|)$ by the estimate $ \int_A 1 d\abs{m} = \abs{m}(A) \leq C\mu(A) = C \int_A 1 d\mu$, for all $A\in\Sigma$, which can be extended using standard Lebesgue integration methods to obtain that for every $f \in L_1(\mu)$ one has $\int \abs{f} d\abs{m} \leq C \int \abs{f} d\mu < \infty$.

    \end{enumerate}
\end{remark}

The proposition below differs from \Cref{def:quasi-invariant-vector-measure} only in that condition \ref{common-condition-quasi-invariant_first_condition} is replaced by \ref{common-condition-quasi-invariant_first_condition_modified}. We present it as a self-contained statement, rather than simply proving the equivalence of \ref{common-condition-quasi-invariant_first_condition} and \ref{common-condition-quasi-invariant_first_condition_modified} under the common assumptions, because it can serve as an alternative definition of $G$-quasi-invariant vector measures.

\begin{proposition} \label{prop:AltDefOfQuasiInvVectorMeasure}
    Let $G$ be a topological group and $X$ be a $G$-Banach space. Suppose that $G$ acts continuously on a finite measure space $(\Omega,\Sigma,\mu)$ by regular set isomorphisms. Let $m: \Sigma \to X$ be an $X$-valued measure on $\Sigma$ with finite total variation such that $|m| \ll \mu$. Then $m$ is $G$-quasi-invariant if and only if:
    \begin{enumerate}[label=(\arabic*$^\star$)]
        \item\label{common-condition-quasi-invariant_first_condition_modified}
            \begin{equation*}
                \forall g \in G : \frac{dg_{*}|m|}{d|m|} = \frac{dg_{*}\mu}{d\mu} \quad \abs{m}\text{-a.e.}
            \end{equation*} where the left hand side of the equality is well-defined.
        \item[(2)]
            \begin{equation*} 
                \forall g \in G \ \forall A \in \Sigma : g \cdot m(A) = T_m \left( \frac{dg_{*} \mu}{d\mu} \chi_{gA} \right).
            \end{equation*}
    \end{enumerate}
\end{proposition}

\begin{proof}
    First, we check that $\frac{dg_*|m|}{d|m|}$ is automatically a well-defined element of $L_1(|m|)$ whenever $m$ is $G$-quasi-invariant. It suffices to check that $g_*|m|\ll|m|$, for every $g\in G$. Fix $g\in G$ and $A\in\Sigma$ such that $|m|(A)=0$. Let $B \in \Sigma$ with $B \subseteq g^{-1} A$. Then $|m|(gB)=0$, i.e., $\chi_{gB}=0$ $|m|$-a.e. By $G$-quasi-invariance of $m$,  
\[
g \cdot m(B) = T_m \left( \frac{dg_*\mu}{d\mu}  \chi_{gB}\right) =0.
\]
Since $g$ acts on $X$ as a linear isometry, it follows that $m(B)=0$, for every $B \subseteq g^{-1} A$. Thus, $g_* |m|(A)=|m| (g^{-1} A) =0$. Consequently, $g_* |m| \ll |m|$.

Next we prove the equivalence between \ref{common-condition-quasi-invariant_first_condition} and \ref{common-condition-quasi-invariant_first_condition_modified}.

The implication \ref{common-condition-quasi-invariant_first_condition_modified} $\implies$ \ref{common-condition-quasi-invariant_first_condition} follows from
    \begin{equation*}
        \int_\Omega \frac{dg_*\mu}{d\mu}\,d|m|
        = \int_\Omega \frac{dg_*|m|}{d|m|}\,d|m|
        = g_*|m|(\Omega)
        = |m|(\Omega)<\infty.
    \end{equation*}

    The more interesting case is to show that \ref{common-condition-quasi-invariant_first_condition} and \ref{common-condition-quasi-invariant_second_condition} imply \ref{common-condition-quasi-invariant_first_condition_modified}. To do so, fix $g \in G$ and define an $X$-valued measure $m_g: \Sigma \to X$ by $m_g(B) = g \cdot m(g^{-1} B)$ for all $B \in \Sigma$. Since the action of $G$ on $X$ is by linear isometries, it is easy to see that
    \begin{equation*}
        \abs{m_g}(B) = \abs{m}(g^{-1}B) = g_*\abs{m}(B)
    \end{equation*}
    for all $B \in \Sigma$, i.e., $\abs{m_g} = g_*\abs{m}$. The assumption (2) gives us, for all $B \in \Sigma$, 
    \begin{equation*}
        m_g(B) = g \cdot m(g^{-1}B) = T_m \left( \frac{dg_*\mu}{d\mu} \chi_{gg^{-1}B} \right) = \int_B \frac{dg_*\mu}{d\mu} dm,
    \end{equation*}
        where the last equality follows by approximating $\frac{dg_*\mu}{d\mu} \chi_B$ by simple functions and using the defining identity $T_m (\chi_A)=m(A)$ together with the continuity of $T_m$.
    By the variation formula for the indefinite integral of a real-valued function with respect to a vector measure (see \cite[Theorem 29 in \S2]{Din2000}), 
    \begin{equation*}
        \abs{m_g}(B) = \int_B \frac{dg_*\mu}{d\mu} d \abs{m}.
    \end{equation*}
    Putting these two ways of expressing $\abs{m_g}$ together, we see that
    \begin{equation*}
        g_*\abs{m}(B) = \abs{m_g}(B) = \int_B \frac{dg_*\mu}{d\mu} d\abs{m}
    \end{equation*}
    for all $B \in \Sigma$, which is the same as saying that $dg_*\abs{m}/d\abs{m} = dg_*\mu/d\mu$.
\end{proof}

The next lemma will be useful when passing from $\mu$ to the variation measure $|m|$.

\begin{lemma} \label{lemma:contActionOnL1(abs(mu))}
    Let $(\Omega, \Sigma, \mu)$ be a finite measure space on which a topological group $G$ acts continuously by regular set isomorphisms. Let $X$ be a $G$-Banach space and $m: \Sigma \to X$ be an $X$-valued measure with finite total variation satisfying $|m|\ll\mu$ such that $m$ is a $G$-quasi-invariant measure. Equip $L_1(\mu)$ and $L_1(|m|)$ with the induced actions of $G$ by positive linear isometries given by \Cref{cor:BanachLamperti}. If the action on $L_1(\mu)$ is continuous, then the action on $L_1(\abs{m})$ is continuous as well.
\end{lemma}
\begin{proof}

    For any $g\in G$, denote by $\widetilde{g}$ the operator on $L_0(\Omega,\Sigma,\mu)$ from \Cref{thm:BanachLamperti}. We claim that for any $g\in G$, 
    \begin{equation}\label{eq:g_tilde_dm_dmu}
    \widetilde{g}\left( \frac{d|m|}{d\mu}\right)=\frac{dg_*|m|}{dg_*\mu}. 
    \end{equation}
    Indeed, for any $A\in\Sigma$ we have \[g_*|m|(A)=|m|(g^{-1}A)=\int_{g^{-1}A} \frac{d|m|}{d\mu}d\mu=\int_A \widetilde{g}\Big(\frac{d|m|}{d\mu}\Big) dg_*\mu,\] which gives the equality.
    
    Next, we claim that
    \begin{equation}\label{eq:dm_dmu_G_invariance}
        \frac{d|m|}{d\mu}=\frac{dg_*|m|}{dg_*\mu} \ \mu\text{-a.e.}
    \end{equation}
    for all $g\in G$. It is enough to check that for every $A\in\Sigma$, $\int_A \frac{d|m|}{d\mu} d\mu= \int_A \frac{dg_*|m|}{dg_*\mu}d\mu$. By the chain rule for Radon-Nikod\'ym derivatives, 
    \[
    \int_A \frac{dg_*|m|}{dg_*\mu}d\mu = \int_A \frac{dg_*|m|}{d|m|}\cdot\frac{d|m|}{d\mu}\cdot \frac{d\mu}{dg_*\mu}d\mu=\int_A \frac{dg_*|m|}{d|m|}\cdot \frac{d\mu}{dg_*\mu}\cdot \frac{d|m|}{d\mu}d\mu,
    \]
    which is, using that $\frac{d\mu}{dg_*\mu}=\frac{d|m|}{dg_*|m|}$, equal to $\int_A \frac{d|m|}{d\mu}d\mu$. This proves the equality in \eqref{eq:dm_dmu_G_invariance}.
    Consequently, by \eqref{eq:g_tilde_dm_dmu} and \eqref{eq:dm_dmu_G_invariance}, we obtain 
    \[
    \widetilde{g}\left(\frac{d|m|}{d\mu}\right)=\frac{d|m|}{d\mu}
    \]
    for all $g \in G$.
    This implies that, modulo $\mu$-null sets, the sets
    \begin{equation*}
        E_n = \left(\frac{d\abs{m}}{d\mu}\right)^{-1} \left( \left[\frac{1}{n},n\right] \right) =  \left\{ \omega \in \Omega: \frac{1}{n} \leq \frac{d\abs{m}}{d\mu} (\omega) \leq n \right\}
    \end{equation*}
    form an increasing sequence of $G$-invariant {$\Sigma$}-measurable sets. Indeed, $E_n$ is $G$-invariant since
    \begin{align*}
    {gE_n} &= {g\left(\frac{d\abs{m}}{d\mu}\right)^{-1} \left( \left[\frac{1}{n},n\right] \right) } \\
    &= \left(\widetilde{g} \left(\frac{d\abs{m}}{d\mu} \right) \right)^{-1} \left( \left[\frac{1}{n},n\right] \right) = \left(\frac{d\abs{m}}{d\mu} \right)^{-1} \left( \left[\frac{1}{n},n\right] \right) =E_n
    \end{align*}
    modulo $\mu$-null sets, where the second equality follows from one of the properties of $\widetilde{g}$ (see \cite[Remarks 3.2.4(v)(c)]{FlemingJamison2002}). 
    Moreover, 
    \[
    \bigcup_{n=1}^\infty E_n = \left \{ \omega \in \Omega : \frac{d\abs{m}}{d\mu} (\omega) > 0 \right \},
    \]
    which is $\abs{m}$-conull.

    Finally, we are ready to prove the continuity of the action of $G$ on $L_1(\abs{m})$. Let $u \in L_1(\abs{m})$ and $\eps > 0$. Find $n \in \N$ such that
    \begin{equation*}
        \norm{u - u\chi_{E_n}}_{L_1(\abs{m})} < \frac{\eps}{3}.
    \end{equation*}
    Then, $u\chi_{E_n} \in L_1(\mu)$ since $\int_{E_n} \abs{u} d\mu \leq n \int_{E_n} \abs{u} d\abs{m} < \infty$.
    By continuity of the action on $L_1(\mu)$, there is a neighbourhood $U \subseteq G$ of $1_G$ such that for all $g \in U$ we have
    \begin{equation*}
        \norm{g \cdot u\chi_{E_n} - u\chi_{E_n}}_{L_1(\mu)} < \frac{\eps}{3n}.
    \end{equation*}
    Then, for all $g \in U$, we have
    \begin{align*}
        &\norm{g\cdot u - u}_{L_1(\abs{m})}\\
        &\hspace{2cm}\leq \norm{g \cdot u\chi_{E_n} - u\chi_{E_n}}_{L_1(\abs{m})} + \norm{g \cdot (u - u\chi_{E_n})}_{L_1(\abs{m})} + \norm{u - u\chi_{E_n}}_{L_1(\abs{m})}\\
        &\hspace{2cm}< n\norm{g \cdot u\chi_{E_n} - u\chi_{E_n}}_{L_1(\mu)} + \frac{2}{3}\eps\\
        &\hspace{2cm}< \frac{\eps}{3} + \frac{2\eps}{3} = \eps.
    \end{align*}
    Note that the second inequality uses the fact that both $g\cdot (u\chi_{E_n})$ and $u\chi_{E_n}$ are supported in $E_n$. This follows from the $G$-invariance of $E_n$ established above.
\end{proof}

Finally, we prove the following result, which shows that the operator $T_m$ is $G$-equivariant if and only if the vector measure $m$ is $G$-quasi-invariant. This equivalence motivates the definition of $G$-quasi-invariance.

\begin{fact}\label{fact:G-quasi-invariant}

Let $(\Omega,\Sigma,\mu)$ be a finite measure space on which a topological group $G$ acts continuously. Let $X$ be a $G$-Banach space and let $m:\Sigma\to X$ be an $X$-valued vector measure of finite variation such that $|m|\ll\mu$ and 
\[
\forall g \in G: \frac{dg_{*}|m|}{d|m|} = \frac{dg_{*}\mu}{d\mu} \quad \abs{m}\text{-a.e.}
\] 
Then $m$ is $G$-quasi-invariant if and only if the associated operator $T_m$ is $G$-equivariant.

\end{fact}

\begin{proof} Suppose first that $m$ is $G$-quasi-invariant. 
We have 
\begin{align*}
    g\cdot T_m (\chi_A) = T_m \left( \frac{dg_{*}\mu}{d\mu}\chi_{gA} \right)  = T_m (g\cdot \chi_A)
\end{align*}
for every $g \in G$ and $A \in \Sigma$.

Since simple functions are dense in $L_1(\Omega, \Sigma, |m|)$ and both $T_m$ and the $G$-actions are continuous, it follows that $g \cdot T_m(f) = T_m(g \cdot f)$ for every $f \in L_1(\Omega, \Sigma, |m|)$, i.e., $T_m$ is $G$-equivariant.

Suppose that $T_m: L_1(\Omega, \Sigma, |m|) \rightarrow X$ is $G$-equivariant. We need to verify \ref{common-condition-quasi-invariant_second_condition} in \Cref{def:quasi-invariant-vector-measure}. Given $g\in G$ and $A\in\Sigma$, we have 
\[g\cdot m(A)=g \cdot T_m (\chi_A)=T_m (g\chi_A)=T_m \left(\frac{dg_*|m|}{d|m|}\chi_{gA}\right)=T_m \left(\frac{dg_*\mu}{d\mu}\chi_{gA}\right). \qedhere
\] 
\end{proof}

While in general we are interested in actions of groups on finite measure spaces $(\Omega,\Sigma,\mu)$ by regular set isomorphisms, in many natural examples the group acts on the underlying set $\Omega$ itself. We explain the relation between such actions and actions by regular set isomorphisms in the remark below.

\begin{remark}\label{rem:RNPdef}
Suppose that $(\Omega,\Sigma,\mu)$ is a finite measure space on which a topological group $G$ acts by a \emph{point action}, that is, $G$ acts on $\Omega$ by measurable bijections that preserve the measure class of $\mu$, i.e., $\mu\ll g_*\mu$ and $g_* \mu \ll\mu$, for all $g\in G$. Then, clearly this action induces an action on $\Sigma$ by regular set isomorphisms.

We note that if $G$ is locally compact and second countable and $(\Omega,\Sigma,\mu)$ is a standard probability space, then every continuous action of $G$ on $(\Omega,\Sigma,\mu)$ by regular set isomorphisms comes from some point action of $G$ on $\Omega$ by measurable bijections that preserve the measure class of $\mu$. This follows from Mackey's theorem (see \cite[Theorem 3.3]{Ramsay71}). Moreover, $(\Omega,\Sigma,\mu)$ may be assumed, without loss of generality, to be a standard probability space if and only if $L_1(\Omega, \Sigma, \mu)$ is separable. One direction is clear. For the other, the separability assumption guarantees that there is a measure algebra isomorphism between $\Sigma/N_\mu$, where $N_\mu$ is the ideal of null sets in $\Sigma$, and a measure algebra of a standard probability space (see e.g. \cite[Theorem 15.3.4]{Roy-book88}). The measure algebra isomorphism induces a linear isometry between the corresponding $L_1$ spaces.
\end{remark}

\subsubsection{Definition of the $G$-RNP}
Finally, we are ready to state the main definitions of this subsection and one of the central new notions of the paper.

\begin{definition}\label{def:newG-RNP}
    Let $G$ be a topological group and $X$ be a $G$-Banach space. We say that $X$ has the \emph{$G$-Radon-Nikodým property} ($G$-RNP, for short) if, for every finite measure space $(\Omega, \Sigma, \mu)$ equipped with a continuous action of $G$ by regular set isomorphisms and for every $G$-quasi-invariant $X$-valued vector measure $m: \Sigma \rightarrow X$, there exists a function $f\in L_1 (\Omega, \Sigma, \mu; X)$ such that
    \begin{equation*}
        m(A) = \int_A f d \mu
    \end{equation*}
    for every $A \in \Sigma$.
\end{definition}

It is immediate from Definition \ref{def:newG-RNP} that if a $G$-Banach space $X$, where $G$ is a topological group, has the RNP, then it has the $G$-RNP.

It is well-known that the classical RNP is equivalent to the RNP with respect to the measure space $([0,1],\mathcal{B},\lambda)$, where $\mathcal{B}$ is the Borel $\sigma$-algebra of $[0,1]$ (\cite[Corollary 8]{DU1977} or \cite[Corollary 2.14]{Pisier2016}). This is also the case for the $G$-RNP when $G$ is locally compact and second countable (see \Cref{cor:G-RNPforinterval}).

Moreover, it is well known that the classical RNP is equivalent to requiring that, for every finite measure space $(\Omega,\Sigma,\mu)$, each $T :L_1 (\Omega,\Sigma,\mu)\to X$ is \emph{representable}, i.e., there exists $\phi \in L_\infty (\Omega, \Sigma, \mu; X)$ such that \begin{equation}\label{eq:G-representable-operators}
        Tf = \int f \phi \, d\mu, \quad \forall f\in L_1 (\Omega,\Sigma,\mu).
    \end{equation}
We shall show that the natural $G$-analogue of this equivalence still holds for any topological group $G$. The following is the $G$-version of the equivalent definition of the RNP in terms of representable operators.

\begin{definition}\label{def:G-representability}
    A $G$-Banach space $X$ has the \emph{$G$-Riesz representation property} if every $G$-equivariant bounded linear operator $T: L_1(\Omega, \Sigma, \mu) \rightarrow X$ is representable, where $(\Omega, \Sigma, \mu)$ is a finite measure space equipped with a continuous action of $G$ by regular set isomorphisms, and $L_1(\Omega, \Sigma, \mu)$ is endowed with the corresponding continuous action of $G$ by positive linear isometries. 
\end{definition}

\begin{example}\label{ex:L1-RNP}
    $L_1 [0,1]$ does not have the $G$-Riesz representation property for any action of any group $G$ on $[0,1]$ by regular set isomorphisms since the identity on $L_1 [0,1]$ (which is, obviously, $G$-equivariant with respect to the induced action by positive linear isometries) does not admit any bounded Radon-Nikod\'ym derivative $\phi : [0,1]\to L_1[0,1]$.
    \end{example}

As promised, we present the equivalence between the $G$-Riesz representation property and the $G$-RNP.

\begin{proposition}\label{prop:G-Riesz-equivalent_to_G-RNP}
Let $G$ be a topological group and $X$ be a $G$-Banach space. The following statements are equivalent. 
  \begin{enumerate}
      \itemsep0.25em
      \item $X$ has the $G$-Riesz representation property.
      \item $X$ has the $G$-RNP.
  \end{enumerate}  
\end{proposition}
\begin{proof}
    Suppose that $X$ satisfies the $G$-Riesz representation property. Let $(\Omega,\Sigma,\mu)$ be a finite measure space such that $G$ acts on it continuously by regular set isomorphisms and let $m: \Sigma \rightarrow X$ be a $G$-quasi-invariant vector measure.
   By \Cref{lemma:contActionOnL1(abs(mu))}, the action of $G$ on $L_1(\Omega, \Sigma, \abs{m})$ is also continuous.
    Moreover, by \Cref{prop:AltDefOfQuasiInvVectorMeasure} and \Cref{fact:G-quasi-invariant}, the action of $G$ on $L_1(\Omega, \Sigma, \abs{m})$ is given by 
    \[
    g \cdot \chi_A = \frac{dg_*|m|}{d|m|}\chi_{gA} = \frac{dg_* \mu}{d\mu}\chi_{gA},
    \]
    and the associated operator $T_m$ is $G$-equivariant. Therefore, by the $G$-Riesz representation property, there exists $\phi \in L_{\infty}(\Omega, \Sigma, \abs{m}; X)$ such that 
    \[
    T_mf = \int_{\Omega} f \phi d\abs{m}
    \]
    for every $f \in L_1(\Omega, \Sigma, \abs{m})$. In particular, for every $A \in \Sigma$, we have that
    \begin{equation*}
        m(A) = T_m(\chi_A) = \int_A \phi d \abs{m}.
    \end{equation*}
    Now, as $\abs{m} \ll \mu$, let $\rho := d\abs{m}/d\mu \in L_1(\Omega,\Sigma,\mu)$ be a nonnegative Radon-Nikod\'ym derivative of $\abs{m}$ with respect to $\mu$.
    
We claim that $\phi \rho$ is (strongly) $\mu$-measurable. By Pettis's measurability theorem, it suffices to prove that $\phi \rho$ is $\mu$-essentially separably valued and weakly $\Sigma$-measurable. First, since $\phi$ is (strongly) $|m|$-measurable, there exists  $N\in\Sigma$ such that $|m|(N)=0$ and $\phi (\Omega \setminus N)$ is a norm separable subset $S$ of $X$. Let $E:= N \cap \{\rho >0\} \in \Sigma$. Note that 
\[
0=|m|(E) = \int_E \rho \,d\mu,
\]
which implies that $\mu (E)=0$.  
If $\omega\notin E$, then either $\omega \in \Omega\setminus N$ or $\rho(\omega)=0$. In the former case,
$\phi(\omega)\in S$, while in the latter case
$(\phi\rho)(\omega)=0$. Hence
\[
(\phi\rho)(\Omega\setminus E)\subseteq \overline{\operatorname{span}}(S).
\]
It follows that $\phi\rho$ is $\mu$-essentially separably valued.
Next, for every $x^*\in X^*$, the scalar-valued function
$x^*\circ\phi$ is $\Sigma$-measurable. Hence $x^*\circ(\phi\rho)=(x^*\circ\phi)\rho$
is $\Sigma$-measurable as a product of scalar-valued measurable functions. This completes the claim.
    Moreover,
    \begin{equation*}
        \int_\Omega \|\phi\rho\|\,d\mu
        = \int_\Omega \|\phi\|\,d|m|
        \leq \|\phi\|_{L_\infty(|m|;X)}|m|(\Omega)<\infty.
    \end{equation*}
    Thus \(\phi\rho\in L_1(\Omega,\Sigma,\mu;X)\) and
    \begin{equation*}
        m(A) = \int_A \phi d \abs{m} = \int_A \phi \rho \, d\mu
    \end{equation*}
    for every $A \in \Sigma$. Hence $\phi\rho$ is a Radon-Nikod\'ym derivative of $m$ with respect to $\mu$. This shows that $X$ has the $G$-RNP.
    
    Conversely, suppose that $X$ has the $G$-RNP. Let $(\Omega,\Sigma,\mu)$ be a finite measure space with a continuous action of $G$ by regular set isomorphisms. Let $T:L_1 (\Omega,\Sigma,\mu)\to X$ be a bounded $G$-equivariant operator. Define the vector measure $m:\Sigma \to X$ by 
    \begin{equation*} 
        m(A):=T(\chi_A) 
    \end{equation*} 
    for every $A \in \Sigma$. Since $\|m(A)\| \leq \|T\| \mu(A)$ for every $A \in \Sigma$, $m$ is a vector measure of finite variation with $|m| \leq \|T\| \mu $. Thus $m$ is dominated by $\mu$. It follows from \Cref{rem:G-quasi-invariance}\ref{rem:QuasiInvVecMeasureInNiceCase} that $m$ satisfies \ref{common-condition-quasi-invariant_first_condition} in \Cref{def:quasi-invariant-vector-measure}. Moreover, \ref{common-condition-quasi-invariant_second_condition} in \Cref{def:quasi-invariant-vector-measure} holds as well since for every $g\in G$ and $A\in\Sigma$, \[g\cdot m(A)=gT(\chi_A)=T(g\chi_A)=T\Big(\frac{dg_*\mu}{d\mu}\chi_{gA}\Big)=T_m\Big(\frac{dg_*\mu}{d\mu}\chi_{gA}\Big),\] where the last equality follows since $L_1(\Omega,\Sigma,\mu)\subseteq L_1 (\Omega,\Sigma,|m|)$ and $T_m\restriction_{L_1(\Omega,\Sigma,\mu)} \equiv T$. Therefore, $m$ is $G$-quasi-invariant.

    Since $X$ has the $G$-RNP, there exists $\phi \in L_1(\Omega, \Sigma, \mu; X)$ such that $m(A) = \int_A \phi d \mu$ for every $A \in \Sigma$. By the standard formula for the variation of a vector measure with Bochner density, $|m|(A) = \int_A \|\phi\| \, d\mu$ (see \cite[Theorem 4, page 46]{DU1977}).  This implies that $\|\phi\| \leq \|T\|$ $\mu$-almost everywhere, thus in fact $\phi \in L_\infty (\Omega, \Sigma, \mu; X)$. Since 
    \[
    T(\chi_A) = m(A) = \int \phi \chi_A \,d\mu
    \]
    for every $A \in \Sigma$, the density of simple functions in $L_1 (\Omega,\Sigma,\mu)$ and continuity show that $T$ is representable, with representing function $\phi$.
    Therefore $X$ has the $G$-Riesz representation property.
\end{proof}

The reader may wonder why the functions $f$ and $\phi$ from \Cref{def:newG-RNP} and \Cref{def:G-representability}, respectively, are not required to have any $G$-equivariance properties. We show that they do have certain $G$-equivariance properties and that such properties are, in fact, automatic. Before introducing these properties let us give some motivation for their definition.

Let $(\Omega,\Sigma,\mu)$ be a finite measure space and $X$ be a Banach space. Suppose that a topological group $G$ acts on $(\Omega,\Sigma,\mu)$ by regular set isomorphisms so that the action comes from a point action on $\Omega$. Then, given $f\in L_1(\Omega,\Sigma,\mu;X)$, we say that $f$ is $G$-equivariant if, as expected, for every $g\in G$ and almost every $\omega\in\Omega$ we have $f(g\omega)=g \cdot f(\omega)$. However, since not every action of $G$ on $\Sigma$ by regular set isomorphisms comes from a point action, we need a more versatile notion of $G$-equivariance.

\begin{definition}\label{def:setisoGequivariance} 
    Let $G$ be a topological group acting on a finite measure space $(\Omega,\Sigma,\mu)$ by regular set isomorphisms and on a Banach space $X$ by linear isometries. We say that a function $f\in L_1(\Omega,\Sigma,\mu;X)$ is \emph{$G$-equivariant-on-sets} if {for every $g \in G$ and every $A \in \Sigma$, the function 
    \begin{equation*}
        \frac{dg_{*} \mu}{d \mu} \chi_{gA} f
    \end{equation*}
    belongs to $L_1(\Omega, \Sigma, \mu; X)$ and
    \begin{equation*} 
    \int_A g\cdot (f(\omega)) d\mu(\omega) = \int_{gA} f(\omega) \frac{d g_*\mu}{d \mu} (\omega)  \, d\mu(\omega).
    \end{equation*}
    }
\end{definition}
The integrability condition on the right-hand side is not automatic and must therefore be included in the definition. Indeed, for a general $f \in L_1(\Omega, \Sigma, \mu; X)$, the function $\frac{dg_{*} \mu}{d \mu} \chi_{gA}f$ need not be Bochner integrable.
Let us show that this indeed generalizes the standard notion of $G$-equivariance.
\begin{fact}
Retain the notation from \Cref{def:setisoGequivariance}. Suppose the action on $(\Omega,\Sigma,\mu)$ comes from a point action on $\Omega$. Then, a function $f\in L_1(\Omega,\Sigma,\mu;X)$ is $G$-equivariant-on-sets if and only if it is $G$-equivariant, that is, $f(g\omega) = g \cdot f(\omega)$ for every $g \in G$ and for almost every $\omega \in \Omega$.
\end{fact}

\begin{proof} {Suppose that $f$ is $G$-equivariant. Fix $g \in G$ and $A \in \Sigma$. First, the function 
\[
        \frac{dg_{*} \mu}{d \mu} \chi_{gA} f
\]
is Bochner integrable. Indeed, by the change of variable formula for the pushforward measure, we get that
\[
\int_{gA} \|f(\omega)\| \frac{dg_{*} \mu}{d\mu}(\omega) d\mu(\omega) = \int_A \|f(g\omega)\| \, d\mu(\omega) = \int_A \|g \cdot (f(\omega)) \| \, d\mu(\omega)  <\infty.
\]
Applying the same change-of-variables formula, we obtain
\begin{equation*}
    \int_A g \cdot (f(\omega)) d \mu(\omega) = \int_{gA} f(\omega) \frac{dg_{*} \mu}{d\mu}(\omega) d\mu(\omega).
\end{equation*}
Thus, $f$ is $G$-equivariant-on-sets.
}

Conversely, suppose that $f$ is $G$-equivariant-on-sets. Fix $g \in G$. For every $A \in \Sigma$, we have that 
\begin{equation*}
    \int_A g \cdot (f(\omega)) d\mu(\omega) = \int_{gA} f(\omega) \frac{dg_{*} \mu}{d\mu}(\omega) d \mu(\omega).
\end{equation*}
By the change of variables formula, the right-hand side is equal to 
\begin{equation*}
    \int_A f(g \omega) d \mu (\omega) 
\end{equation*}
and, therefore,
\begin{equation*}
    \int_A (g \cdot (f(\omega)) - f(g \omega)) d \mu(\omega) = 0 
\end{equation*}
for every $A \in \Sigma$. By the uniqueness for Bochner integrable functions, we obtain that $g \cdot (f(\omega)) = f(g\omega)$ for $\mu$-almost every $\omega \in \Omega$. Since $g$ was arbitrary, $f$ is $G$-equivariant.
\end{proof}

Now we are ready to show that the functions $f$ and $\phi$ from \Cref{def:newG-RNP} and \Cref{def:G-representability}, respectively, satisfy this definition of $G$-equivariance.

\begin{proposition} \label{prop:newG-RNP_function_f} Let $X$ be a $G$-Banach space. Let $(\Omega, \Sigma, \mu)$ be a finite measure space and suppose that $G$ acts on $(\Omega, \Sigma, \mu)$ by regular set isomorphisms. Let $m: \Sigma \rightarrow X$ be a $G$-quasi-invariant vector measure. Suppose that 
\begin{equation*}
    m(A) = \int_A f d \mu 
\end{equation*}
for every $A \in \Sigma$, where $f \in L_1(\Omega, \Sigma, \mu; X)$. Then, $f$ is $G$-equivariant-on-sets. In particular, every Radon-Nikodým derivative appearing in the definition of the $G$-RNP is $G$-equivariant-on-sets. Likewise, every representing function appearing in the definition of the $G$-Riesz representation property is $G$-equivariant-on-sets.
\end{proposition}

\begin{proof}   
By assumption, the variation of $m$ is given by 
\begin{equation*}
    |m|(B) = \int_B \|f(\omega)\| d\mu(\omega) 
\end{equation*}
for every $B \in \Sigma$. That is, $d|m| = \|f(\cdot)\| d\mu$. Since $m$ is $G$-quasi-invariant, we have that ${dg_{*} \mu} /{d\mu} \in L_1(\Omega, \Sigma, |m|)$. Thus, for $g \in G$ and $A \in \Sigma$, we obtain that 
\begin{equation*}
    \int_{gA} \frac{dg_{*} \mu}{d\mu}(\omega) \|f(\omega)\| d \mu(\omega) < \infty.
\end{equation*}
This implies that $\frac{dg_{*} \mu}{d\mu} \chi_{gA} f \in L_1(\Omega, \Sigma, \mu; X)$ and then the integral 
\begin{equation*}
    \int_{gA} f(\omega) \frac{dg_{*} \mu}{d\mu}(\omega) d\mu(\omega) 
\end{equation*}
is well-defined. Now, on the one hand, we have that 
\begin{equation*}
    g \cdot m(A) = g \cdot \int_A f(\omega) d\mu(\omega) = \int_A g \cdot ( f(\omega) ) \,d\mu(\omega).
\end{equation*}
On the other hand, we have that 
\begin{equation*}
    T_m \left( \frac{dg_{*} \mu}{d\mu} \chi_{gA} \right) = \int_{\Omega} \frac{dg_{*} \mu}{d\mu}(\omega) \chi_{gA}(\omega)dm(\omega) = \int_{gA} f(\omega) \frac{dg_{*} \mu}{d\mu}(\omega) d\mu(\omega) 
\end{equation*}
using the identity $dm = fd\mu$. Using the $G$-quasi-invariance of $m$, we get that 
\begin{equation*}
    \int_A (g \cdot f(\omega)) \, d \mu(\omega) = \int_{gA} f(\omega) \frac{dg_{*} \mu}{d\mu}(\omega) d\mu(\omega).
\end{equation*}
Since $g \in G$ and $A \in \Sigma$ are arbitrary, we get that $f$ is $G$-equivariant-on-sets. 

Finally, suppose that $T: L_1(\Omega, \Sigma, \mu) \rightarrow X$ is a $G$-equivariant bounded linear operator represented by some $\phi \in L_{\infty}(\Omega, \Sigma, \mu; X)$. Define $m_T(A):= T(\chi_A)$ for every $A \in \Sigma$. Then, $m_T$ is $G$-quasi-invariant by the equivalence between $G$-quasi-invariant vector measures and $G$-equivariant operators (see the proof of \Cref{prop:G-Riesz-equivalent_to_G-RNP}). Moreover, 
\begin{equation*}
    m_T(A) = \int_A \phi \,d\mu 
\end{equation*}
for every $A \in \Sigma$. Applying the first part of the proof to $m_T$, we conclude that $\phi$ is $G$-equivariant-on-sets.
\end{proof}

We conclude this section by emphasizing that \Cref{prop:G-Riesz-equivalent_to_G-RNP} allows one to identify the $G$-RNP with the $G$-Riesz representation property. Consequently, several arguments in the remainder of the paper will be formulated in terms of the $G$-Riesz representation property, even when the stated goal is to establish the $G$-RNP. We shall therefore freely identify these two notions without further comment.

\section{Consequences of the $G$-Bishop-Phelps property}\label{section:G-BP} The aim of this section is to prove that, as in the classical case (see \cite[Proposition 1]{Bourgain1977}), the $G$-Bishop-Phelps property implies strong $G$-dentability.
However, we prove it under the assumption that $G$ is a compact group (see \Cref{remark:G-dentability}(c) for a discussion on why compactness of $G$ is required).

The following is the main result of this section.

\begin{theorem}\label{thm:G-BP_implies_G-dentable-main} Let $G$ be a compact group and $X$ be a $G$-Banach space. If $X$ satisfies the $G$-BP, then $X$ is strongly $G$-dentable.
\end{theorem}

In order to obtain \Cref{thm:G-BP_implies_G-dentable-main}, it suffices to prove \Cref{thm:G-BP_implies_G-dentable}, which deals with \emph{separable} $G$-invariant subsets, along with the separable reduction argument provided in \Cref{lem:G-dent-separablydetermined}.

\begin{theorem}\label{thm:G-BP_implies_G-dentable}
Let $G$ be a compact group, $X$ be a $G$-Banach space, and let $C\subseteq X$ be a separable bounded closed convex $G$-invariant set. If $C$ has the $G$-BP, then it is $G$-dentable.
\end{theorem}

Before proceeding to the proof of \Cref{thm:G-BP_implies_G-dentable}, we need several preparatory lemmas.

\begin{lemma}\label{lem:G.separable}
Let $G$ be a compact group, $X$ be a $G$-Banach space, and let $A\subseteq X$ be a separable subset. Then $G\cdot A$ is separable.
\end{lemma}

Notice that we {\it do not} require that the group $G$ is separable.

\begin{proof}
Let $A'\subseteq A$ be a countable dense subset. Then, clearly $G\cdot A'$ is dense in $G\cdot A$. Since $G\cdot a$ is compact for every $a\in A'$, it follows that $G\cdot A'$ is $\sigma$-compact, therefore separable. Consequently, $G\cdot A$ is separable.
\end{proof}

The following argument generalizes the classical case from \cite{May74}.

\begin{lemma}\label{lem:G-dent-separablydetermined}
Let $G$ be a compact group and $X$ be a $G$-Banach space. Let $C\subseteq X$ be a $G$-invariant subset. If every separable $G$-invariant subset of $C$ is $G$-dentable, then $C$ is itself $G$-dentable.
\end{lemma}
\begin{proof}
Suppose that $C$ is not $G$-dentable. Then, there exists $\varepsilon>0$ such that, for every finite tuple $x_1,\ldots,x_n\in C$, we have $C \subseteq \overline{\co}_G\big(C\setminus B_\varepsilon(x_1,\ldots,x_n)\big)$. Therefore, for every $x\in C$ and every finite tuple $x_1,\ldots,x_n\in C$, there exists a separable set $E^x_{\{x_1,\ldots,x_n\}}\subseteq C\setminus B_\varepsilon(x_1,\ldots,x_n)$ such that \[x\in \overline{\co}_G (E^x_{\{x_1,\ldots,x_n\}})=\overline{\co} (G\cdot E^x_{\{x_1,\ldots,x_n\}}).
\] Notice that by \Cref{lem:G.separable}, the set $G\cdot E^x_{\{x_1,\ldots,x_n\}}$ is separable and it is also $G$-invariant.

Now given any separable $G$-invariant subset $E$ of $C$ with a countable dense subset $\{e_n\}_{n\in\N}$, we define a set $\hat E$, omitting the countable dense subset from the notation, as follows. First let $(\vec e_n)_{n\in\N}$ be an enumeration of all finite tuples of $\{e_n\}_{n\in\N}$.  Set \[\hat E:=\bigcup_{i\in\N,\; j\in\N} G\cdot E^{e_i}_{\vec e_j}.\] Notice that $\hat E$ is still separable and $G$-invariant.

We recursively define $G$-invariant separable sets $D_n$ increasing with respect to inclusion, $n\geq 0$. Pick any $x\in C$ and set $D_0:=\{g\cdot x\colon g\in G\}$. Suppose that $D_n$ has been defined. Choose a countable dense subset $\{d^n_i\}_{i\in\N}$ of $D_n$ that contains $\{d^{n-1}_i\}_{i\in\N}$ when $n\geq 1$, and consider the set $\hat D_n$ with respect to $\{d^n_i\}_{i\in\N}$. Then, we define 
\[D_{n+1}:=\hat D_n \cup D_n.
\] 
Proceeding in this way, we obtain $D_0 \subseteq\ldots\subseteq D_n\subseteq \ldots$.

Finally, we set $\mathsf{D}:=\bigcup_n D_n$, which is a  $G$-invariant and separable subset of $C$ as a countable union of $G$-invariant separable sets. Set also $\mathcal{D}:=\{d_i^j\colon i,j\in\N\}$ which is a countable dense subset of $\mathsf{D}$. Let us check that $\mathsf{D}$ is not $G$-dentable, which will yield the desired contradiction. This is witnessed by $\varepsilon/2$. 
Suppose there exist a finite tuple $x'_1,\ldots,x'_n\in \mathsf{D}$ and $x\in\mathsf{D}$ such that $x\notin \overline{\co}_G\big(\mathsf{D}\setminus B_{\varepsilon/2}(x'_1,\ldots,x'_n)\big)$. Since $\mathsf{D}\setminus \overline{\co}_G\big(\mathsf{D}\setminus B_{\varepsilon/2}(x'_1,\ldots,x'_n)\big)$ is relatively open in $\mathsf{D}$, we may suppose that $x\in\mathcal{D}$. Moreover, we can find $x_1,\ldots,x_n\in \mathcal{D}$ such that $B_{\varepsilon/2}(x'_1,\ldots,x'_n)\subseteq B_\varepsilon(x_1,\ldots,x_n)$, and thus $x\notin \overline{\co}_G\big(\mathsf{D}\setminus B_\varepsilon(x_1,\ldots,x_n)\big)$. Since there is $m\in\N$ such that $x,x_1,\ldots,x_n\in \{d_i^m\}_{i\in\N}$, by definition $E^x_{\{x_1,\ldots,x_n\}}\subseteq D_{m+1}\subseteq \mathsf{D}$, we get  
\[
x\in\overline{\co}_G (E^x_{\{x_1,\ldots,x_n\}}) \subseteq \clco_G (\mathsf{D} \setminus B_{\varepsilon} (x_1,\ldots,x_n)),
\] which is the promised contradiction.
\end{proof}

\begin{corollary}\label{cor:weakerstrongG-dent}
Let $G$ be a compact topological group. A $G$-Banach space $X$ is strongly $G$-dentable if and only if every separable bounded closed absolutely convex $G$-invariant set is $G$-dentable.
\end{corollary}
\begin{proof}
One implication is obvious. Suppose that every separable bounded closed absolutely convex $G$-invariant set is $G$-dentable and let $C\subseteq X$ be an arbitrary bounded $G$-invariant set. By \Cref{lem:G-dent-separablydetermined}, it suffices to consider the case where $C$ is separable. Suppose, toward a contradiction, that $C$ is not $G$-dentable. Then obviously $D:=C\cup (-C)$ is not $G$-dentable either. By \Cref{fact:sGdentabilityInheritance}, $E:=\clco(D)$ is also not $G$-dentable. On the other hand, $E$ is separable, bounded, closed, absolutely convex, and $G$-invariant, contradicting the assumption. 
\end{proof}

We also need the theorem below implied by the Big Peter-Weyl theorem. Given a topological group $G$ acting by linear isometries on a Banach space $X$, we say that $x \in X$ is \emph{almost $G$-invariant} (or \emph{$G$-finite}) if $\Span (G\cdot x)$ is a finite-dimensional space. The set of all almost $G$-invariant elements of $X$ will be denoted by $X_\mathrm{fin}$. 

\begin{theorem}[A version of the Big Peter-Weyl Theorem, {\cite[Theorem 3.51]{HofMorbook}}]\label{thm:PW}
    Let $G$ be a compact group and $X$ be a $G$-Banach space. Then, we have that
    $X= \overline{X_{\mathrm{fin}}}$. 
\end{theorem}

Now we are ready to prove \Cref{thm:G-BP_implies_G-dentable}. The argument is inspired by the proof of \cite[Proposition~1]{Bourgain1977}.
It is worth emphasizing, however, that our proof diverges substantially from Bourgain’s original approach: in our $G$-equivariant framework the use of the Peter–Weyl theorem is essential, as it provides sufficiently many almost $G$-invariant elements that serve as the key ingredients for the renorming argument in the proof. 

\begin{proof}[Proof of \Cref{thm:G-BP_implies_G-dentable}] 
Let us assume that $C\subseteq B_X$ and that it is not $G$-dentable. By \Cref{fact:equivDefOfStrongGDentability}, there exists $\varepsilon>0$ such that, for every finite set $\{x_1,\ldots, x_n\}$ in $X$, we have 
\[
C \subseteq \clco_G (C \setminus B_\varepsilon (x_1,\ldots,x_n) ). 
\]
Let $Y\subseteq X$ be the separable closed subspace generated by $C$, which is $G$-invariant. Thanks to \Cref{thm:PW}, we can choose a countable dense sequence $(y_n)_{n\in\N}\subseteq Y$ such that for every $n\in\N$ the space
\[
Y_n:=\overline{\Span}\big\{g \cdot y_n\colon g\in G\big\}
\] is finite-dimensional. We define a new norm $\vertiii{\cdot}$ on $X$ as follows  
\begin{equation*} 
\vertiii{x}:=\sqrt{\|x\|^2+\sum_{n=1}^\infty \frac{\dist(x,Y_n)^2}{2^{n}}},\quad x\in X.
\end{equation*} 
It is straightforward to check that $\vertiii{\cdot}$ is an equivalent norm on $X$ which is still $G$-invariant, i.e., the action of $G$ on $(X,\vertiii{\cdot})$ is still implemented by linear isometries, since each $Y_n$ is $G$-invariant.

Consider the identity operator $I:(X,\|\cdot\|)\to (X,\vertiii{\cdot})$. Since $I$ is $G$-equivariant and $C$ has the $G$-Bishop-Phelps property, there exists a $G$-equivariant bounded linear operator $T:(X,\|\cdot\|)\to(X,\vertiii{\cdot})$ which attains the maximum on $C$ at an element $x\in C$ and satisfies $\|I-T\|<\varepsilon/4$. Choose $n\in\N$ such that $\dist(x,Y_n)<\varepsilon/4$. Let $z_1,\ldots,z_m\in Y_n$ be an $\varepsilon/8$-net in $(1+\varepsilon)B_{Y_n}$ noting that $B_{Y_n}$ is compact since $Y_n$ is finite-dimensional.

By the choice of $\varepsilon>0$, we obtain that 
\begin{equation*} 
C \subseteq \overline{\co}_G(C\setminus B_\varepsilon(z_1,\ldots,z_m)).
\end{equation*} 
Set $D:=B_\varepsilon(z_1,\ldots,z_m)$. We claim that 
\[
\vertiii{2 T(x)}=\sup_{y\in G\cdot(C\setminus D)} \vertiii{T(x+y)}.
\] 
Indeed, since $C \subseteq \overline{\co}_G(C\setminus D)$, for every $\delta>0$, there exist $x_1,\ldots,x_k\in C\setminus D$, $g_1,\ldots,g_k\in G$, and $\lambda_1,\ldots,\lambda_k \geq 0$ such that $\sum_{i=1}^k \lambda_i=1$ and $\|x-\sum_{i=1}^k \lambda_i g_i x_i\|<\delta$. Hence, there exists $i \in \{1,\ldots,k\}$ such that 
\[ 
\big| \vertiii{  2T(x)} - \vertiii{T( x+g_i \cdot x_i )} \big|   <\delta\|T\|,
\]
which completes the claim, since $\delta>0$ was arbitrary.

Choose sequences $(u_k)_{k\in\N}\subseteq C\setminus D$ and $(g_k)_{k\in\N}\subseteq G$ satisfying 
\begin{equation}\label{eq:vertiii1}
    \vertiii{2T(x)}=\lim_{k\to\infty} \vertiii{T(x+g_k \cdot u_k)}.
\end{equation} 
Note that for every $z\in X$, $\vertiii{z}$ is precisely the $\ell_2$-norm of 
\[
\left( \|z\|, \frac{\dist (z, Y_1)}{2^{1/2}}, \frac{\dist (z, Y_2)}{2}, \ldots,\frac{\dist (z, Y_n)}{2^{n/2}},\ldots  \right)
\]
It is well-known, for $a\in\ell_2$, $(b_k)_k\subseteq \ell_2$, that 
\begin{equation}\label{eq:vertiii2}
\|a+b_k\|_{\ell_2} \to 2\|a\|_{\ell_2} \quad \text{and} \quad \|b_k\|_{\ell_2} \leq \|a\|_{\ell_2}\quad\text{implies}\quad    \|b_k-a\|_{\ell_2} \to 0. 
\end{equation}

It follows from \eqref{eq:vertiii1} and \eqref{eq:vertiii2} that 
\begin{equation*} 
\dist(T(x),Y_n)=\lim_{k\to\infty} \dist(T(g_k \cdot u_k),Y_n)
\end{equation*}
for every $n\in\mathbb{N}$.
However, notice that for every $y\in C\setminus D$, we have $\dist(y,Y_n)\geq \frac{7\varepsilon}{8}$. In fact, otherwise there would be $y'\in Y_n$ such that $\|y-y'\|<7\varepsilon/8$. In particular, this implies that $\|y'\| \leq \|y-y'\| +\|y\| <1+\varepsilon.$ Thus, there exists $i \in \{1,\ldots,m\}$
such that $\|y'- z_i\|<\varepsilon/8$, thus $y\in B_\eps(z_i)\subseteq D$, a contradiction. Therefore, we obtain that 
\begin{align*}
\dist(T(x),Y_n) &\leq \|T(x)-x\|+\dist(x,Y_n) \\
&\leq \vertiii{T(x)-x}+\dist(x,Y_n) \leq \frac{\e}{4} + \frac{\e}{4} = \frac{\e}{2} 
\end{align*}
which yields the following contradiction
\begin{align*}
\dist(T(g_k u_k),Y_n) &= \dist(T(u_k),Y_n) \\
&\geq \dist(u_k,Y_n)-\|T(u_k)-u_k\| \\
&\geq \dist (u_k,Y_n)-\vertiii{T(u_k)-u_k} 
\geq \frac{7\varepsilon}{8}-\frac{\varepsilon}{4}=\frac{5\varepsilon}{8}.
\end{align*}
As a consequence, $C$ is $G$-dentable.
\end{proof}

\begin{proof}[Proof of \Cref{thm:G-BP_implies_G-dentable-main}]
By \Cref{cor:weakerstrongG-dent}, it is enough to check that every separable bounded closed absolutely convex $G$-invariant set $C$ is $G$-dentable. By assumption, $C$ has $G$-BP, so by \Cref{thm:G-BP_implies_G-dentable}, it is $G$-dentable.
\end{proof}

We conclude this section with the following result. We note that we do not know how to prove it directly; instead, the argument requires a substantial detour through the results established above.

\begin{corollary}\label{cor:G-BPimpliesBP} Let $G$ be a compact group and $X$ be a $G$-Banach space. If $X$ has the $G$-BP, then $X$ has the BP.
\end{corollary}

\begin{proof}  Assume that $X$ has the $G$-BP. Since $G$ is compact, $X$ is strongly $G$-dentable by \Cref{thm:G-BP_implies_G-dentable-main}. So $X$ is weakly $G$-dentable which implies being dentable by \Cref{prop:weakly-G-dentable-and-dentable-are-equivalent}. Hence, by the classical result \cite[Theorem 7]{Bourgain1977}, $X$ has the BP.
\end{proof}

 \section{Consequences of $G$-dentability}\label{section:G-dent}

In the classical setting, dentability is equivalent to the Radon–Nikodým and Bishop–Phelps properties, and moreover, it implies the Krein–Milman property. In this section, we show that strong $G$-dentability implies the $G$-KMP for compact $G$, and that weak $G$-dentability implies the $G$-RNP for any $G$, in the two respective subsections. It remains open whether strong or weak $G$-dentability implies the $G$-BP for any non-trivial $G$.

\subsection{Strong $G$-dentability implies $G$-KMP} \label{subsection:strongly-G-dentable-G-KMP} In this subsection, our main goal is to connect the strong $G$-dentability and the $G$-Krein-Milman property of a $G$-Banach space $X$. More specifically, we will prove the following result.

\begin{theorem} \label{theorem:strong-G-dentable-implies-G-KMP} Let $G$ be a compact group and $X$ be a $G$-Banach space. If $X$ is strongly $G$-dentable, then $X$ has the $G$-KMP.
\end{theorem}

Let us emphasize once again that the assumption of compactness of the group $G$ arises from the fact that strong $G$-dentability is only meaningful when $G$ is (pre)compact (see \Cref{fact:strongG-dent}). To prove \Cref{theorem:strong-G-dentable-implies-G-KMP}, we will show that if $D$ is a nonempty closed bounded convex subset of a $G$-Banach space $X$ such that every closed bounded convex $G$-invariant subset of $\clco_G (D)$ is $G$-dentable, then $D$ must contain an extreme point (see \Cref{proposition:existence-of-ext-pts} below). The proof of this last fact draws on ideas from a classical argument due to J. Lindenstrauss establishing that every nonempty closed bounded convex subset of a Banach space with the RNP contains an extreme point (see, for instance, \cite[Theorem 5, p.~190]{DU1977}). More precisely, we will prove the following result.

\begin{proposition} \label{proposition:existence-of-ext-pts} 
    Let $D$ be a~nonempty convex closed bounded subset of a $G$-Banach space $X$. Assume that every convex closed bounded $G$-invariant subset of $\clco_G(D)$ is $G$-dentable. Then, $D$ has an~extreme point.
\end{proposition}

\Cref{theorem:strong-G-dentable-implies-G-KMP} follows as a consequence of \Cref{proposition:existence-of-ext-pts}. Indeed, assume that
\Cref{proposition:existence-of-ext-pts} holds and let $X$ be a strongly
$G$-dentable Banach space. Suppose, towards a contradiction, that $X$ does not have the $G$-KMP. That is, there exists a bounded convex closed $G$-invariant set $D \subseteq X$ such that
\begin{equation*} \label{KM-1}
D \neq \overline{\co}_G(\ext(D))
\end{equation*}
where $\ext(D)$ denotes the subset of extreme points of $D$.

Pick any $x'_0 \in D \setminus \overline{\co}_G(\ext(D))$. By the Hahn-Banach separation theorem applied to $x'_0$ and $\overline{\co}_G(\ext(D))$ we obtain $(x')^*\in X^*$ separating $x'_0$ and $\overline{\co}_G(\ext(D))$. By applying the Bishop-Phelps theorem to $(x')^*$, there exist $x^* \in X^*$ and $x_0\in D\setminus \overline{\co}_G(\ext(D))$ such that
\[
x^*(x_0) = \sup x^*(D) > \sup x^*(\overline{\co}_G(\ext(D))).\]
Consider the subset $F:= \{x \in D: x^*(x) = x^*(x_0) \}$ which is a nonempty convex closed bounded subset of $D$. Since $X$ is assumed to be strongly $G$-dentable, in particular, every convex closed bounded $G$-invariant subset of $\clco_G (F)$ is $G$-dentable. Applying \Cref{proposition:existence-of-ext-pts}, we conclude that $F$ has an extreme point. Let $e$ be an extreme point of $F$. We claim that $e$ is an extreme point of $D$. Indeed, suppose that there are $u, v \in D$ and $\lambda \in [0,1]$ such that $e = \lambda u + (1 - \lambda) v$. Then we have that
\begin{equation*}
x^*(x_0) = x^*(e) = \lambda x^*(u) + (1-\lambda) x^*(v).
\end{equation*}
Since $x^*(x_0) = \sup x^*(D)$, we get that $x^*(u) = x^*(v) = x^*(x_0)$. This implies that $u,v \in F$ and, therefore, $u=v=e$. So, $e \in \ext (D)$. Consequently,
\[
x^* (e) = \sup x^* (D) > \sup x^*(\overline{\co}_G(\ext(D))) \geq x^* (e),
\]
which is a contradiction.

Therefore, in order to prove \Cref{theorem:strong-G-dentable-implies-G-KMP}, it remains to prove \Cref{proposition:existence-of-ext-pts}. For this we will use the following well-known result.

\begin{lemma} [\mbox{Kuratowski's intersection theorem, \cite[(f), p.~6]{Measuresofnoncpt-book}}] \label{lma:intersectionOfAlmostCpcts}
    Let $D_1 \supseteq D_2 \supseteq \cdots$ be a~decreasing sequence of closed subsets of a~complete metric space $M$. If there is a~sequence $(\eps_n)_{n=1}^\infty \subseteq (0, \infty)$ such that $\eps_n \to 0$ and, for every $n \in \N$, the set $D_n$ has a~finite $\eps_n$-net, then $\bigcap_{n \in \N} D_n$ is a~nonempty compact set.
\end{lemma}

\begin{proof}[Proof of \Cref{proposition:existence-of-ext-pts}] We will inductively construct a decreasing sequence $D_0 \supseteq D_1 \supseteq \cdots$ of nonempty closed convex subsets of $D$ such that $D_n$ has a finite $2^{-n}$-net for every $n\geq1$. Put $D_0 = D$.
    
    Assume that $D_n \subseteq D$ has already been constructed. Since $D_n$ is bounded, so is $\clco_G (D_n)$. Clearly, $\clco_G (D_n) \subseteq \clco_G (D)$. Thus, by the assumption, $\clco_G (D_n) = \clco (G\cdot D_n)$ is $G$-dentable. It follows from \Cref{fact:sGdentabilityInheritance} that $G \cdot D_n$ is also $G$-dentable. Choosing $\eps_n = 2^{-n-2}$, it follows that there exist $m_n \in\N$, $x^n_1, \dots, x^n_{m_n} \in G \cdot D_n$, and $x^n \in G\cdot D_n$ such that
    \begin{equation*}
        x^n \notin \clco_G \big( G \cdot D_n \setminus B_{\eps_n}(x^n_1, \cdots, x^n_{m_n}) \big) =: C_n.
    \end{equation*}
    If $C_n = \emptyset$, then $D_n \subseteq G \cdot D_n \subseteq B_{\eps_n}(x^n_1, \cdots, x^n_{m_n})$, which implies that $D_n$ has a~finite $2^{-n-1}$-net and thus we can set $D_{n+1} = D_n$.
    
    So, assume that $C_n \neq \emptyset$. Since $C_n$ is $G$-invariant, we may, without loss of generality, assume that $x^n \in D_n$. By the Hahn-Banach theorem, there exists $x^*_n \in X^*$ such that 
    \[
    x^*_n(x^n) > \sup x^*_n (C_n).
    \]
    Since $D_n$ is a bounded closed convex set, we may invoke the Bishop-Phelps theorem to find $y^*_n \in X^*$ such that $y^*_n$ attains its maximum on $D_n$ at, say, $y^n \in D_n$ and $y_n^*$ is sufficiently close to $x_n^*$ so that 
    \begin{equation*}
        y^*_n(y^n) = \sup y_n^* (D_n) > \sup y^*_n (C_n).
    \end{equation*}
    Define $M_n := y^*_n(y^n)$ and
    \begin{equation*}
        D_{n+1} := \left\{
            x \in D_n : y^*_n(x) = M_n
        \right\}.
    \end{equation*}
    By the definition of $M_n$, the sets $D_{n+1}$ and $C_n$ must be disjoint. However, $D_{n+1} \subseteq D_n \subseteq G \cdot D_n$ and hence, by the definition of $C_n$, we have that $D_{n+1} \subseteq B_{\eps_n}(x^n_1, \cdots, x^n_{m_n})$, implying that $D_{n+1}$ has a~finite $2^{-n-1}$-net. This implies in turn that $D_{n+1} \subseteq D_n$ is a nonempty closed convex set with a~finite $2^{-n-1}$-net.
    
    Define $K := \bigcap_{n\in\mathbb{N}} D_n$. Using \Cref{lma:intersectionOfAlmostCpcts}, we observe that $K$ is a~nonempty compact set. Moreover, since all $D_n$ are convex, so is $K$. Hence, by the classical Krein-Milman theorem, $K$ has an~extreme point, call it $z \in K$.
    
    We claim that $z$ is an~extreme point of $D$.
    To see this, let $z = \alpha z_1 + (1-\alpha)z_2$ with $\alpha \in [0,1], z_1, z_2 \in D$. We prove by induction that $z_1, z_2 \in D_n$ for all $n \geq 0$. The base case $z_1, z_2 \in D_0 = D$ is trivial. Assume that $z_1, z_2 \in D_n$ for some $n \geq 0$. Since $z \in K \subseteq D_{n+1}$,
    \begin{equation*}
        M_n = y^*_n(z) = \alpha y^*_n(z_1) + (1-\alpha)y^*_n(z_2).
    \end{equation*}
    As $M_n$ is the maximal value of $y^*_n$ on $D_n$, it follows that $y^*_n(z_1) = y^*_n(z_2) = M_n$, implying that $z_1, z_2 \in D_{n+1}$ and completing the induction step. Consequently, $z_1, z_2 \in K$; and since $z \in \mathrm{ext}(K)$, we obtain that $z_1 = z_2 = z$. This proves that $z$ is indeed an~extreme point of $D$.
\end{proof}

\subsection{Weak $G$-dentability implies the $G$-RNP} 
In this subsection, we establish a connection between weak $G$-dentability and the $G$-RNP of a $G$-Banach space. Notice from \Cref{prop:weakly-G-dentable-and-dentable-are-equivalent} and the discussion following \Cref{def:newG-RNP} that if a $G$-Banach space $X$ is weakly $G$-dentable, then $X$ is dentable, and hence has the $G$-RNP. 

Nevertheless, for separable groups $G$, we present a direct proof that weak $G$-dentability implies the $G$-RNP (see \Cref{theorem:weak-G-dentable-implies-G-RNP} below). We conclude the section by showing that there is no $G$-equivariant version of Lewis-Stegall's factorization theorem (see \Cref{example:Lewis-Stegall}).

\begin{theorem} \label{theorem:weak-G-dentable-implies-G-RNP}
    Let $G$ be a separable topological group and $X$ be a $G$-Banach space. If $X$ is weakly $G$-dentable, then $X$ has the $G$-RNP.
\end{theorem}

\begin{proof} By \Cref{prop:G-Riesz-equivalent_to_G-RNP}, it is enough to prove that $X$ has the $G$-Riesz representation property.
Let $(\Omega,\Sigma,\mu)$ be a finite measure space and $T: L_1 (\Omega,\Sigma,\mu) \to X$ be a $G$-equivariant bounded linear operator, where $G$ acts on $L_1 (\Omega,\Sigma,\mu)$ by positive linear isometries. By the Banach-Lamperti theorem, we can suppose that the action of $G$ on $L_1(\Omega,\Sigma,\mu)$ is induced by an action of $G$ on $\Sigma/N(\mu)$ (see \Cref{cor:BanachLamperti}).

Given $A \in \Sigma$ with $\mu(A)>0$, consider 
\begin{align*}
    &C_A := \left\{ T \left(\frac{\chi_B}{\mu (B)} \right) : B \subseteq A, \, \mu(B)>0 \right\}
\end{align*} 
and set
\begin{align*}
    D_A := \overline{\{ Tf : f \geq 0, \, \text{supp}(f)\subseteq A, \, f \in S_{L_1 (\Omega, \Sigma,\mu)} \}}.
\end{align*}
For simplicity, let us write $x_B := T(\chi_B)/ \mu(B)$. 

 Notice that, for every $A\in\Sigma$, we have
 \begin{enumerate}
     \item $C_A\subseteq  D_A$;
     \item in fact, by approximation of positive $L_1$-functions by positive simple functions, we have that $D_A=\overline{\mathrm{co}}(C_A)$;
     \item in particular, $\diam D_A=\diam C_A$; \label{item:diam} 
     \item while for $g\in G$ we do not in general have $gC_A=C_{gA}$, we do have $gD_A=D_{gA}$. \label{item:gDA} 
 \end{enumerate}

 Let now $G'\leq G$ be a countable dense subgroup, which exists since $G$ is separable.

 \noindent\textbf{Claim 1.} Let $A\in\Sigma$, with $\mu(A)>0$, be $G'$-invariant. Then, $D_A$ is $G$-invariant.

Indeed, by item \eqref{item:gDA} above, for every $g \in G'$, we have $gD_A = D_{gA} = D_A$ as $A$ is $G'$-invariant, so $D_A$ is $G'$-invariant. Now, let $g \in G$ and $x \in D_A$. Choose a net $(g_{\alpha})_{\alpha}$ in $G'$ converging to $g$. By continuity of the action and since $D_A$ is closed, $gx=\lim_\alpha g_{\alpha}x\in D_A$, showing that $D_A$ is $G$-invariant.

\noindent\textbf{Claim 2.} For any $G'$-invariant set $A \in \Sigma$ with $\mu(A)>0$ and any $\varepsilon>0$, there exists $B \subseteq A$ with $\mu(B)>0$ and $\diam (C_B) \leq \varepsilon$. 

Assume to the contrary that there exists a $G'$-invariant set $A \in \Sigma$ with $\mu(A) > 0$ and $\varepsilon>0$ such that, for every $B \subseteq A$ with $\mu(B)>0$, we have $\diam  (C_B) > \varepsilon$. We then show that the $G$-invariant set $D_A$ is not dentable, which contradicts the assumptions. Since $D_A=\overline{\mathrm{co}}(C_A)$, note that $D_A$ is dentable if and only if $C_A$ is dentable \cite{Phelps1974, Rieffel1967}. So it is enough to show that $C_A$ is not dentable. In fact, we will show that no point of $C_A$ is dented with radius $\e/2$, that is, for a fixed $B \subseteq A$ with $\mu(B)>0$, we shall show that $x_B \in \overline{\co}(C_A \setminus B_{\e/2}(x_B))$.

To show this, we use an exhaustion argument. Take a maximal disjoint family $\{E_n\}$ of measurable subsets of $B$ such that $\mu(E_n)>0$ for every $n \in \N$ and that for any $n \in \mathbb{N}$, $\| x_B - x_{E_n} \| > \varepsilon/2$. Put $E= B \setminus (\cup_n E_n)$. We claim that $\mu(E)=0$. If not, by the assumption, $\diam  (C_E)>\varepsilon$. Thus, there exists $y \in  C_E$ such that $\| x_B - y \| >\varepsilon/2$. Take $E'\subseteq E$ so that $\mu(E')>0$ and $y=x_{E'}$. The existence of $E'$ contradicts the maximality of $\{E_n\}$. Consequently, we obtain that $B = \cup_n E_n$ (up to null sets). Therefore,
\begin{equation*}
x_B = \sum_n \frac{\mu(E_n)}{\mu(B)} x_{E_n}.
\end{equation*} 
The above equality is understood as convergence in norm. Equivalently, it follows by applying $T$ to the $L_1$-convergent identity
\begin{equation*}
    \frac{\chi_B}{\mu(B)} = \sum_n \frac{\mu(E_n)}{\mu(B)} \frac{\chi_{E_n}}{\mu(E_n)}.
\end{equation*}
For every $n$, we have $x_{E_n} \in C_A \setminus B_{\e/2}(x_B)$ by construction. Then, we have $x_B \in \overline{\co}(C_A \setminus B_{\e/2}(x_B))$. Since this holds for every $B \subseteq A$ with $\mu(B) > 0$, $C_A$ is not dentable. This completes the proof of Claim 2.

Thanks to Claim 2, we can consider, for given $\varepsilon>0$, a maximal disjoint family $\{A_n\}$ of measurable subsets of $\Omega$ such that $\mu(A_n) >0$ and $\diam ( C_{A_n})\leq \varepsilon$ for every $n \in \mathbb{N}$.

\noindent \textbf{Claim 3.} The set $\Omega$ is equal to $\bigcup_n A_n$ up to measure zero. 

Otherwise, set $E:=\Omega\setminus \bigcup_n A_n$ and assume that $\mu(E)>0$. If there exists a $G'$-invariant subset $E'\subseteq E$ of positive measure, then by {Claim 2}, there is $E''\subseteq E'$ such that $\mu(E'')>0$ and $\diam C_{E''}\leq \varepsilon$, which contradicts the maximality of $\{A_n\}$. This implies that the set 
\[
E \, \setminus 
\bigcup_{\substack{n \in \mathbb{N}, \, g \in G'}} 
g\cdot A_n
\]
is $\mu$-null. Since $G'$ is countable, $\mu(E)>0$, and from the $\sigma$-additivity of $\mu$, there exist $n\in\N$ and $g\in G'$ such that $\mu(gA_n\cap E)>0$. Setting $F:=A_n\cap g^{-1}E$, 
\begin{equation*}
\diam C_{gA_n\cap E}=\diam D_{gA_n\cap E} \xlongequal{\text{\eqref{item:diam}, \eqref{item:gDA} }} \diam C_F \leq \diam C_{A_n}\leq \varepsilon.
\end{equation*}
This again contradicts the maximality of $\{A_n\}$ and completes the proof of Claim 3.

Finally, define $\phi_\varepsilon :\Omega \to X$ by $\phi_\varepsilon(\omega) = x_{A_n}$ if $\omega \in A_n$ for some $n\in\mathbb{N}$. Let $f \in S_{L_1 (\Omega,\Sigma,\mu)}$ be a positive function and let $\alpha_n = \int_{A_n} f \, d\mu$ for every $n \in \mathbb{N}$. In the following estimate, we omit the terms for which $\alpha_n = 0$. Note that, whenever $\alpha_n > 0$, we have
 \[
T\left( \frac{f \chi_{A_n} }{\alpha_n} \right) \in  D_{A_n} \quad \text{ and } \quad x_{A_n} \in C_{A_n}.
\]
Thus,
\begin{align*} 
\left\|Tf - \int f  \phi_\varepsilon \, d\mu \right\| &= \left\| \sum_{\alpha_n > 0} \left( T(f \chi_{A_n}) - \int_{A_n} f \phi_\varepsilon \, d\mu \right) \right\| \\ 
&\leq \sum_{\alpha_n > 0} \alpha_n \left\| T\left( \frac{f \chi_{A_n} }{\alpha_n} \right) - x_{A_n} \right\| \leq \varepsilon
\end{align*} 
where the last inequality holds since $\diam (D_{A_n})\leq \varepsilon$. Therefore, it follows that $\|Tf- \int f \phi_\varepsilon \, d\mu\| \leq \varepsilon \|f\|$ for all $f \in L_1 (\Omega,\Sigma,\mu)$.

We repeat this procedure with $\varepsilon_k = 1/2^k$, i.e., construct the maximal family $\{ A_n^k\}$ as above, and consider the corresponding $\widetilde{\phi}_k := \phi_{\varepsilon_k} \in L_\infty (\Omega, \Sigma,\mu; X)$ satisfying that $\|Tf - \int f \widetilde{\phi}_{k} \, d\mu \| \leq \|f\|/2^k$ for every $f \in L_1 (\Omega,\Sigma,\mu)$ and $k \in \mathbb{N}$. We claim that $\|\tilde{\phi}_n - \tilde{\phi}_m\|_{\infty} \leq 2^{-n} + 2^{-m}$. Indeed, fix $n,m \in \N$. For almost every $\omega \in \Omega$, there are indices $i,j$ such that $\omega \in A_i^n \cap A_j^m$ and $\mu(A_i^n \cap A_j^m) > 0$. Put $B:= A_i^n \cap A_j^m$. Since $B \subseteq A_i^n$ and $B \subseteq A_j^m$, we have $x_B \in C_{A_i^n} \cap C_{A_j^m}$. Therefore,
\begin{equation*}
\|x_{A_i^n} - x_{A_j^m}\| \leq \|x_{A_i^n} - x_B\| + \|x_B - x_{A_j^m}\| \leq \diam C_{A_i^n} + \diam C_{A_j^m} \leq 2^{-n} + 2^{-m}.
\end{equation*}
Since $\tilde{\phi_n}(\omega) = x_{A_i^n}$ and $\tilde{\phi}_m (\omega)=x_{A_j^m}$ for such $\omega$, this completes the claim. Thus, $(\tilde{\phi}_k)$ converges in $L_{\infty}(\Omega, \Sigma, \mu; X)$ to a function $\phi$ and $Tf = \int f \phi \, d \mu$ for every $f \in L_1(\Omega, \Sigma, \mu)$. This proves that $T$ is representable.
\end{proof}

Recall from Lewis-Stegall's result (see, for instance, \cite[Theorem 3, page 114]{DU1977}) that a bounded linear operator $T: L_1 (\Omega,\Sigma,\mu) \to X$ is representable if and only if $T$ factors through $\ell_1$, i.e., 
\[ \xymatrix{
		L_1 (\Omega,\Sigma,\mu) \ar[rd]_S \ar[rr]^T  & & X  \\
			& \ell_1 \ar[ur]_{L}  &
		} \]
This factorization result leads to a natural question whether every $G$-equivariant representable operator $T : L_1 (\Omega,\Sigma,\mu) \to X$ (see \Cref{def:G-representability}) factors through $\ell_1$ in a \emph{$G$-equivariant sense}, i.e., $T=L\circ S$, where $L$ and $S$ both are $G$-equivariant operators. However, the following example shows that this is not the case. 

 \begin{proposition} \label{example:Lewis-Stegall}
    There is no $G$-equivariant Lewis-Stegall's factorization theorem.
 \end{proposition}
 
 \begin{proof}
Consider $G=\Iso (\mathbb{R}^2)$ and $(B_{\mathbb{R}^2}, \Sigma,\mu)$, where $\mu$ is the Lebesgue measure on $B_{\mathbb{R}^2}$. Note that $G$ is precisely the orthogonal group $O(2)$. Consider the subgroup $SO(2)$ of $O(2)$ consisting of the orientation-preserving linear isometries of $\R^2$. Define 
\[
T:L_1 (B_{\mathbb{R}^2}, \Sigma,\mu) \to \mathbb{R}^2, \quad T(\chi_A) = \int_A x \, d\mu(x) \text{ for every $A \in \Sigma$.}
\]
 
 Notice that $T$ is $G$-equivariant. {Moreover, $T$ is represented by the Bochner measurable function $x\mapsto x$ defined on $\R^2$}. 
 
 Suppose, for contradiction, that $T$ factors through $G$-equivariant operators 
$L:\ell_1 \to \mathbb{R}^2$ and $S:L_1 (B_{\mathbb{R}^2}, \Sigma,\mu) \to \ell_1$. 
Note that $\Iso(\ell_1)$ is the group of signed permutations, that is, 
\[
\Iso (\ell_1 ) = \{ (x_n) \mapsto (\varepsilon_n x_{\sigma(n)}) : (\varepsilon_n) \in \{-1,1\}^{\mathbb{N}}, \, \sigma : \mathbb{N}\to\mathbb{N} \text{ bijection} \}.
\]
In particular, $\Iso (\ell_1)$ is totally disconnected. Hence the restriction to $SO(2)$ of every continuous homomorphism
\[
G \to \Iso (\ell_1)
\]
is trivial, since $SO(2)$ is the connected component of the identity. Therefore, the restriction of any continuous action of $G$ on $\ell_1$ to $SO(2)$ is trivial.

Since every $r \in SO(2)$ acts as $\id_{\ell_1}$ on $\ell_1$ and since $L$ is $G$-equivariant, for every $x \in \ell_1$ and every $r \in SO(2)$, we have that 
\[
r \cdot L(x) = L(r x) = L(x).
\]
However, the only vector of $\R^2$ fixed by $SO(2)$ is $0$. Thus, $L(x) = 0$ for every $x \in \ell_1$ and so $L=0$. This implies that $T = L \circ S = 0$, which is impossible.
\end{proof}

It is worth mentioning that the space in \Cref{example:Lewis-Stegall} has the $G$-RNP since $\mathbb{R}^2$ has the RNP and the RNP implies the $G$-RNP (see \Cref{def:newG-RNP} above).

 \section{Consequences of the $G$-Radon-Nikod\'ym property}\label{section:G-RNP}

The main goal of this section is to prove that the $G$-Radon–Nikodým property implies weak $G$-dentability for groups $G$ satisfying a certain measure condition, which holds, for instance, for locally compact and $\sigma$-compact groups. Along the way, we develop the basic theory of $G$-equivariant-on-sets martingales. We begin with the definition of the class of groups considered in this section.

\begin{definition} \label{def:propertyF}
Let $G$ be a topological group. We say that $G$ has \emph{property $\FF$} if there exists a finite measure space $(\Omega,\Sigma,\mu)$ such that $G$ admits a point action on $\Omega$ (recall \Cref{rem:RNPdef}) satisfying the following conditions:
\begin{itemize}
    \item the corresponding action on $(\Omega,\Sigma,\mu)$ by regular set isomorphisms is continuous;
    \item the action on $\Omega$ is \emph{essentially free}, i.e., for every $g\in G\setminus\{1_G\}$ \[\mu\big(\{\omega\in\Omega\colon g\cdot \omega=\omega\}\big)=0,\]
    \item it admits a \emph{measurable fundamental domain}, i.e., there is $F\in\Sigma$ such that $\bigcup_{g\in G} gF=\Omega$ up to measure zero and for every $g\neq h\in G$, $gF\cap hF=\emptyset$.
    \item there exist a conull $G$-invariant set $\Omega_0 \subseteq \Omega$ and a measurable map $\sigma:\Omega_0 \rightarrow G$ such that, for every $\omega \in \Omega_0$, one has $\omega \in \sigma(\omega)F$, where $F$ is the measurable fundamental domain from the previous item.
\end{itemize}
\end{definition}

\begin{example}\label{ex:propertyF} 
Every locally compact $\sigma$-compact group $G$ has property $\FF$. Indeed, let $\mu_G$ be the Haar measure on $G$ (we refer to \cite{HofMorbook} for Haar measures on locally compact groups). Set $\Omega=G$ and let $\Sigma$ be the Borel $\sigma$-algebra on $G$. If $G$ is compact, then set $\mu=\mu_G$ and we are done. Otherwise, $\mu_G$ is not finite; however, by $\sigma$-compactness, it is $\sigma$-finite. Let $(K_n)_{n\in\N}$ be a sequence of pairwise disjoint measurable sets of strictly positive finite measure such that $G=\bigcup_{n\in\N} K_n$ and let $(\alpha_n)_{n\in\N}\subseteq \R$ be a sequence of strictly positive reals such that $\sum_{i=1}^\infty \alpha_i=1$. Then, define \[
\widetilde{\mu_G} (A):=\sum_{i=1}^\infty \alpha_i  \frac{\mu_G(A\cap K_i)}{\mu_G(K_i)},\quad A\in\Sigma.
\] It is straightforward to check that $(\Omega,\Sigma,\widetilde{\mu_G})$ is then a probability measure space witnessing that $G$ has property $\FF$ (the fundamental domain being the singleton $\{1_G\}$). The measurable orbit-coordinate map required in Definition \ref{def:propertyF} is simply $\sigma(g) = g$ for every $g \in G$.
\end{example}

In order to show that the $G$-RNP implies weak $G$-dentability for a certain group $G$, we shall need an auxiliary notion of a $G$-equivariant-on-sets martingale. While it is of independent interest, in the present context it will only play the role of an intermediary between the $G$-RNP and weak $G$-dentability.

\begin{definition} Let $G$ be a topological group and $X$ be a $G$-Banach space.
Let $(\Omega,\Sigma,\mu)$ be a finite measure space and suppose $G$ acts continuously on $(\Omega,\Sigma,\mu)$ by regular set isomorphisms. A sequence $(M_n)_{n\geq 0}\subseteq L_1(\Omega,\Sigma,\mu;X)$ is a \emph{$G$-equivariant-on-sets martingale} if each $M_n$, $n\geq 0$, individually is $G$-equivariant-on-sets (recall \Cref{def:setisoGequivariance}) and $(M_n)_{n\geq 0}$ is a martingale with respect to a filtration of $G$-invariant sub-$\sigma$-algebras of $\Sigma$. That is, there is a sequence  $\A_0\subseteq\A_1\subseteq\ldots\subseteq\Sigma$ such that $\A_n$ is $G$-invariant, each $M_n$ is $\A_n$-measurable, and $\mathbb{E}^{\A_n}(M_{n+1})=M_n$, where $\mathbb{E}^{\A_n}$ is the corresponding conditional expectation. Here, $G$-invariance of a sub-$\sigma$-algebra $\mathcal{A}$ of $\Sigma$ is understood in the sense that $g \cdot A \in \mathcal{A}$ for all $g \in G$ and $A \in \mathcal{A}$. 
\end{definition}

We refer to \cite[Section 1.3]{Pisier2016} for details on the classical definition of Banach-space-valued martingales.

\begin{theorem}\label{thm:G-RNPimpliesG-dentability}
Let $G$ be a topological group and $X$ be a $G$-Banach space. Consider the following conditions.
	\begin{enumerate}
		\item $X$ has the $G$-RNP.
		\item For every finite measure space $(\Omega, \Sigma, \mu)$ equipped with a~continuous action of $G$ by regular set isomorphisms, every $X$-valued $G$-equivariant-on-sets martingale which is uniformly bounded in $L_\infty$ (i.e. $\sup_{n\in\N} \norm{M_n}_{L_\infty} < \infty$) converges a.e. and in $L_1(\Omega,\Sigma,\mu; X)$.
		\item $X$ is weakly $G$-dentable.
	\end{enumerate} 
	Then, we have that $(1) \implies (2)$ and $(3) \implies (1)$ for any $G$, and  $(2) \implies (3)$ if $G$ has property $\FF$.
 In particular, if $G$ is locally compact and $\sigma$-compact, then all the conditions are equivalent.
\end{theorem}

\begin{proof}
$(3)\implies (1)$ has already been established in the discussion preceding \Cref{theorem:weak-G-dentable-implies-G-RNP}.

$(1) \implies (2)$: The proof is the same as the proof of $(i)\implies (ii)$ in~\cite[Theorem~2.9]{Pisier2016}, to which we refer the reader. Here we only show how the weaker hypothesis of $(1)$ is sufficient to get the weaker conclusion $(2)$.

Let $(M_n)$ be an~$X$-valued uniformly bounded $G$-equivariant-on-sets martingale. Since $(M_n)$ is uniformly bounded in $L_\infty(\Omega,\Sigma,\mu; X)$, it is clear that $(M_n)$ is uniformly integrable in $L_1(\Omega,\Sigma,\mu; X)$. Thus, the proof of $(i)\implies (ii)$ from~\cite[Theorem~2.9]{Pisier2016} shows that the formula
	\begin{equation*}
		m(A) = \lim_{n \to \infty} \int_A M_n d\mu \quad (A \in \Sigma)
	\end{equation*}
defines a countably additive measure of finite total variation which is absolutely continuous with respect to $\mu$.

In order to be able to apply the $G$-RNP instead of the RNP in the proof of $(i)\implies (ii)$ from \cite[Theorem 2.9]{Pisier2016}, we only need to check that $m$ is $G$-quasi-invariant and then the remainder of the proof is the same as in \cite{Pisier2016}. Since we have the quite strong assumption of $(M_n)$ being bounded in $L_\infty(\Omega,\Sigma,\mu; X)$, we shall show that $\abs{m}$ is dominated by $\mu$ and then, by \Cref{rem:G-quasi-invariance}\ref{rem:QuasiInvVecMeasureInNiceCase}, we will only need to verify \ref{common-condition-quasi-invariant_second_condition} in \Cref{def:quasi-invariant-vector-measure}.

To check that $\abs{m}$ is dominated by $\mu$, notice that for every finite measurable partition $\{A_1, \dots, A_N\}$ of a set $A \subset \Omega$, we have
\begin{align*}
    \sum_{i=1}^N \norm{m(A_i)}_X
    &= \lim_{n \to \infty} \sum_{i=1}^N \norm{\int_{A_i}M_n d\mu}_X
    \leq \limsup_{n \to \infty} \sum_{i=1}^N \int_{A_i}\norm{M_n} d\mu\\
    &\leq \left( \sup_{n \in \N} \norm{M_n}_{L_\infty(\mu; X)} \right) \mu(A).
\end{align*}
Taking the supremum over all such partitions, we obtain 
\[
\abs{m} \leq \left( \sup_{n \in \N} \norm{M_n}_{L_\infty(\mu; X)} \right)\mu.
\]

Since $m$ has finite total variation, the operator $T_m : L_1 (\Omega,\Sigma,|m|) \to X$ is well-defined. Since $L_1 (\Omega,\Sigma,\mu) \subseteq L_1 (\Omega,\Sigma,|m|)$ (see \Cref{rem:G-quasi-invariance}\ref{rem:QuasiInvVecMeasureInNiceCase}), the restriction $T_m \restriction_{L_1 (\Omega,\Sigma,\mu)}$ is well-defined. To show $G$-quasi-invariance of $m$ we only need to check that, for every $g \in G$ and every $A \in \Sigma$, we have
\begin{equation*}
    g \cdot m(A) = T_m \left( \frac{dg_{*} \mu}{d\mu} \chi_{gA} \right).
\end{equation*}

To do so, fix $g \in G$ and $A \in \Sigma$. Observe that 
	\begin{align*}
		g \cdot m(A)
		&= \lim_{n \to \infty} \int_A g \cdot M_n (\omega) \, d\mu (\omega)
		= \lim_{n \to \infty} \int_{gA} M_n (\omega) \frac{d g_*\mu}{d\mu} (\omega) \, d\mu(\omega)
	\end{align*}
    since $M_n$ is $G$-equivariant-on-sets. 
	For ease of notation, denote
	\begin{equation*}
		h = \frac{d g_*\mu}{d\mu}.
	\end{equation*}

	Note that $h$ is nonnegative; thus 
	\begin{equation*}
		\norm{h}_1 = 
     \int_\Omega \frac{d g_* \mu}{d\mu} \ d\mu = \mu(g^{-1}\Omega) = \mu(\Omega) < \infty,
	\end{equation*}
	i.e., $h \in L_1 (\Omega,\Sigma,\mu)$. Since simple functions are dense in $L_1 (\Omega,\Sigma,\mu)$, we can find a monotone sequence $(h_k)_{k\in\mathbb{N}}$ of nonnegative simple functions such that $h_k \to h$ in $L_1 (\Omega,\Sigma,\mu)$ and $h_k \nearrow h$ a.e. By definition, there exist a sequence $(n_k)_{k\in\N}$, a sequence of families $(D^k_i)_{1\leq i\leq n_k}$ of disjoint subsets of $\Omega$, and families of nonnegative scalars $(\alpha^k_i)_{1\leq i \leq n_k}$ such that 
    \[
    h_k = \sum_{i=1}^{n_k} \alpha^k_i \chi_{D^k_i}.
    \]
     For each $k \in \mathbb{N}$, denote by 
	\begin{equation}\label{eqn:def_of_Ik}
		I_k = \lim_{n \to \infty} \int_{gA} h_k M_n d\mu.
	\end{equation}
	First, we need to show that the limit in \eqref{eqn:def_of_Ik} exists. In fact, we have that 
	\begin{equation}
		\label{eqn:GRNPtoMartingale:IkLinearity}
		\begin{split}
			\int_{gA} h_k M_n d\mu = \int_{gA} \sum_{i=1}^{n_k} \alpha^k_i \chi_{D^k_i} M_n d\mu &= \sum_{i=1}^{n_k} \alpha^k_i \int_{{gA} \cap D^k_i} M_n d\mu \\ 
			&\xrightarrow[n\to\infty]{\text{}} \sum_{i=1}^{n_k} \alpha^k_i m(gA \cap D^k_i) \in X
		\end{split}
	\end{equation}
        by definition of the vector measure $m$.
    
	Our goal now is to show that $\lim_{k \to \infty} I_k = g \cdot m(A)$. We start by choosing $\eps > 0$ arbitrarily. Let $K > 0$ be such that for all $n \in \N$ and almost every $\omega \in \Omega$, we have $\norm{M_n(\omega)} < K$. Next, choose $k_0 \in \N$ such that $\norm{h - h_k} < \eps/K$ for all $k\geq k_0$. Then, for all such $k$, we have
	\begin{align*}
		\norm{\int_{gA} h M_n d\mu - \int_{gA} h_k M_n d\mu}
		&\leq \int_{gA} \norm{(h - h_k) M_n} d\mu \\
		&\leq \int_{gA} K\abs{h - h_k} d\mu
		< \eps.
	\end{align*}
	Letting $n \to \infty$, we obtain that $\norm{I_k - g \cdot m(A)} \leq \eps$. Since $\eps > 0$ was arbitrary, it follows that $I_k \to g\cdot m(A)$.
	
	We now analyze the quantities $I_k$. Note from \eqref{eqn:GRNPtoMartingale:IkLinearity} that 
	\begin{equation*}
		I_k
		= \sum_{i=1}^{n_k} \alpha^k_i m(gA \cap D^k_i)
		= \sum_{i=1}^{n_k} \alpha^k_i T_m(\chi_{gA \cap D^k_i})
		= T_m(h_k \chi_{gA}).
	\end{equation*}
	Since we assume that $h_k$ converge a.e. to $h$ from below, the dominated convergence theorem yields that $h_k\chi_{gA} \to h\chi_{gA}$ in $L_1 (\Omega,\Sigma,\mu)$. Since $T_m$ is a~bounded operator on $L_1 (\Omega,\Sigma,\mu)$, it follows that $T_m(h_k\chi_{gA}) \to T_m(h \chi_{gA})$.
	
	Now we have all the pieces ready; it only remains to put them together to finish the proof of the $G$-quasi-invariance of $m$. Indeed, 
	\begin{align*}
		g\cdot m(A)
		= \lim_{k \to \infty} I_k
		= \lim_{k \to \infty} T_m(h_k \chi_{gA})
		= T_m (h \chi_{gA})
		= T_m \left( \frac{d g_*\mu}{d\mu} \chi_{gA} \right).
	\end{align*}
Thus, $m$ is a dominated $G$-quasi-invariant vector measure. Since $X$ has the $G$-RNP, there exists $\varphi \in L_{\infty}(\Omega, \Sigma, \mu; X)$ such that
\begin{equation*}
    m(A) = \int_A \varphi d\mu
\end{equation*}
for every $A \in \Sigma$. This is precisely the vector-measure representation needed in the classical proof of \cite[Theorem~2.9]{Pisier2016}. Therefore, the rest of the argument gives that the martingale $(M_n)$ converges almost everywhere and in $L_1(\Omega, \Sigma, \mu; X)$ to $\varphi$. This proves (1) $\implies$ (2).

$(2)\implies (3)$: Suppose that $G$ has property $\FF$ which is witnessed by an appropriate action on a finite measure space $(\Omega',\Sigma',\mu')$. Let $F'\subseteq \Omega'$ be the corresponding measurable fundamental domain. Let $\Omega_0' \subseteq \Omega'$ be the conull $G$-invariant set and let $\sigma': \Omega_0' \rightarrow G$ be the measurable orbit-coordinate map provided by Definition \ref{def:propertyF} so that $\omega' \in \sigma'(\omega') F'$ for every $\omega' \in \Omega_0'$. We define the finite measure space $(\Omega,\Sigma,\mu)$ as follows: 
\begin{itemize}
\item $\Omega := \Omega' \times [0,1]$;
\item $\Sigma$ is the product $\sigma$-algebra on $\Omega$ generated by $\Sigma'$ on $\Omega'$ and the Borel $\sigma$-algebra on $[0,1]$;
\item $\mu := \mu' \times \lambda$, where $\lambda$ denotes the Lebesgue measure on $[0,1]$.
\end{itemize}
 Moreover, we define a point action $\alpha_G$ of $G$ on $\Omega$  by
 \[
 \alpha_G (g, (\omega', t)) := \big(\alpha_{G}' (g,\omega'), t\big), \quad g\in G,\,  (\omega',t)\in \Omega' \times [0,1],
 \]
 where $\alpha'_{G} : G \curvearrowright \Omega'$ is the group action witnessing property $\FF$ of $G$. In other words, $G$ acts on the first coordinate as the original action and trivially on the second.  Clearly, the corresponding action on $(\Omega,\Sigma,\mu)$ by regular set isomorphisms is continuous. We shall use the conull $G$-invariant set $\Omega_0 = \Omega_0' \times [0,1]$. As before, we will omit the notation $\alpha_G$ or $\alpha_{G}'$ whenever the context makes the meaning clear.

To reach a contradiction, assume that $X$ is not weakly $G$-dentable. Then, there is a~$G$-invariant bounded set $D\subseteq X$ and $\delta > 0$ such that for every $x \in D$, we have $x \in \overline{\co}(D \setminus B_\delta(x))$. By a standard argument (see e.g. \cite[Lemma 2.7]{Pisier2016}), we may in fact assume that, for all $x \in D$, we have $x \in \co(D \setminus B_\delta(x))$.

Fix an arbitrary $x_0 \in D$ and define $M_0: \Omega \to D$ by 
\begin{equation*}
    M_0(\omega',t):= \sigma'(\omega') \cdot x_0
\end{equation*}
for $(\omega', t) \in \Omega_0$ and extend it arbitrarily on the null set $\Omega \setminus \Omega_0$.

Clearly, $M_0$ is well-defined on a conull set. Moreover, it is $G$-equivariant and measurable with respect to the sub-$\sigma$-algebra $\A_0:=\sigma\big(\{B\times [0,1]\colon B\in\Sigma'\}\big)$.

Since $x_0 \in \co(D \setminus B_\delta(x_0))$, there exist $n \in \N$, $t_1, \dots, t_n \in \left(0,1\right]$, and $x_1, \dots, x_n \in D$ such that
\[
    x_0 = \sum_{i=1}^n t_i x_i,
    \quad
    \sum_{i=1}^n t_i = 1,
    \quad \text{and} \quad
    \norm{x_i - x_0} \geq \delta, \,\, \forall i=1,\ldots,n.
\]
	We can then find disjoint measurable subsets $A_1, \dots, A_n \subseteq [0,1]$ such that $\lambda(A_i) = t_i$ and $\bigcup_{i=1}^n A_i = [0,1]$ and define 
\begin{equation*}
M_1(\omega',t):=\sigma'(\omega') \cdot x_i \quad \text{if $(\omega',t) \in \Omega_0$ and $t \in A_i$.}
\end{equation*}

    It is straightforward that $M_1$ is well-defined up to a null set and also $G$-equivariant. Moreover, it is measurable with respect to the sub-$\sigma$-algebra $\A_0\subseteq\A_1$, where $\A_1:=\sigma\big(\{B\times A_i\colon B\in\Sigma',1 \leq i\leq n\}\big)$.

    Furthermore, observe that, for a.e. $(\omega',t)\in\Omega$, we have 
    \[
    \|M_0(\omega',t)-M_1(\omega',t)\| \geq \delta.\] 
Indeed, let $\omega = (\omega',t) \in \Omega_0$. Then, there exists $i \in \{1,\ldots,n\}$ such that $t \in A_i$, so 
    \[
    \|M_0(\omega)-M_1(\omega)\|=\|\sigma'(\omega')\cdot x_0-\sigma'(\omega') \cdot x_i\|=\|x_0-x_i\| \geq \delta.
    \] 
    In the next step, for each $1\leq i\leq n$, since $x_i\in\co(D\setminus B_\delta(x_i))$, there exist $n_i\in\N$, $t^i_1,\ldots,t^i_{n_i}\in \left(0,1\right]$ and $x^i_1,\ldots,x^i_{n_i}\in D$ such that \[
		x_i = \sum_{j=1}^{n_i} t^i_j x^i_j,
		\quad
		\sum_{j=1}^{n_i} t^i_j = 1,
		\quad \text{and} \quad
		\norm{x^i_j - x_i} \geq \delta \,\, \text{for all } j=1,\ldots,n_i.
\]
Then we can find disjoint measurable sets $A^i_1,\ldots,A^i_{n_i}\subseteq A_i$ such that $\lambda(A^i_j)= t^i_j\lambda(A_i)$ and $\bigcup_{j=1}^{n_i} A^i_j=A_i$. Define 
\begin{equation*}
    M_2(\omega', t):= \sigma'(\omega') \cdot x_j^i
\end{equation*}
if $(\omega',t) \in \Omega_0$ and $t \in A_j^i$ for some $1 \leq i \leq n$ and $1 \leq j \leq n_i$. As before, $M_2$ is well-defined up to a null set, $G$-equivariant and measurable with respect to the sub-$\sigma$-algebra $\A_0\subseteq\A_1\subseteq \A_2$, where $\A_2:=\sigma\big(\{B\times A^i_j\colon B\in\Sigma',1\leq i\leq n, 1\leq j\leq n_i\}\big)$. Again as above, we check that, for a.e. $\omega\in\Omega$, we have 
\[
\|M_1(\omega)-M_2(\omega)\| \geq \delta.
\] 
Continuing in the same way, we produce a $G$-equivariant-on-sets martingale $(M_n)_{n\in\N}$ such that
\begin{itemize}
    \item $\|M_n(\omega)-M_{n+1}(\omega)\| \geq \delta$ for all $n\geq 0$ and 
    \item $\sup_{n\in\N} \|M_n\|_{L_\infty}\leq \sup_{x\in D} \|x\|$, so the martingale is uniformly bounded.
\end{itemize}
However, the first condition certainly prevents the martingale from converging a.e., which yields a contradiction. 
\end{proof}

Observe that if $G$ is locally compact  and second countable, then it is $\sigma$-compact, and thus the finite (probability) measure space $(G, \Sigma_G, \widetilde{\mu_G})$ witnesses the property $\mathcal{F}$ of $G$ (\Cref{ex:propertyF}).

\begin{corollary}\label{cor:G-RNPforinterval}
Let $G$ be a locally compact and second countable group. Then a $G$-Banach space $X$ has the $G$-RNP if and only if it has the $G$-RNP only with respect to the standard probability space $([0,1],\mathcal{B},\lambda)$.
\end{corollary}

\begin{proof}
One implication is trivial. For the reverse implication assume that $X$ has the $G$-RNP with respect to the standard probability space $([0,1],\mathcal{B},\lambda)$. Note that the argument in the proof of $(2)\implies (3)$ of \Cref{thm:G-RNPimpliesG-dentability} can be adapted, since in this case it actually relies only on the $G$-RNP with respect to the measure space $(\Omega,\Sigma,\mu)$, where $\Omega=G\times [0,1]$, $\Sigma$ is the Borel $\sigma$-algebra of $G\times [0,1]$, and $\mu = \widetilde{\mu_G} \times\lambda$ with $\widetilde{\mu_G}$ a measure equivalent to the Haar measure on $G$ and $\lambda$ the Lebesgue measure on $[0,1]$.

Notice that $\mu = \widetilde{\mu_G} \times\lambda $ is a probability measure that is non-atomic. Since $G$ is locally compact and second countable, it is Polish, so $G \times [0,1]$, endowed with its Borel $\sigma$-algebra, is a standard Borel space. Therefore, $(\Omega, \Sigma, \mu)$ is a non-atomic standard probability space. Since $(\Omega,\Sigma)$ is standard Borel, $(\Omega,\Sigma,\mu)$ and $([0,1],\mathcal{B},\lambda)$ are isomorphic as measure spaces by \cite[Theorem 17.41]{Kechris}. It follows that $X$ has the $G$-RNP with respect to $(G \times [0,1],\Sigma,\mu)$, and by applying $(2)\implies(3)$ of \Cref{thm:G-RNPimpliesG-dentability}, we conclude that $X$ is weakly $G$-dentable. Finally, $(3)\implies(1)$ of \Cref{thm:G-RNPimpliesG-dentability} shows that $X$ has the $G$-RNP.
\end{proof}

We conclude with the following result, which identifies the classical RNP with the $G$-RNP for the particularly well-behaved class of groups considered in this section. We emphasize that the equivalence is not at all obvious from the definitions, and our proof is obtained indirectly through several intermediate implications between related properties.

\begin{corollary}\label{cor:G-RNPisRNP}  Let $G$ be a locally compact and second countable group. Then, a $G$-Banach space $X$ has the $G$-RNP if and only if it has the RNP.
\end{corollary}

\begin{proof}
One implication is immediate and holds for every group $G$. Conversely, suppose that $X$ has the $G$-RNP. As observed just before \Cref{cor:G-RNPforinterval}, the group $G$ has property $\mathcal{F}$. Therefore, \Cref{thm:G-RNPimpliesG-dentability} implies that $X$ is weakly $G$-dentable. Hence, by \Cref{prop:weakly-G-dentable-and-dentable-are-equivalent}, $X$ is dentable. Therefore $X$ has the RNP (see, for instance, \cite[Chapter VII, Section 6]{DU1977}).
\end{proof}

 \section{Norm-attaining $G$-equivariant operators}\label{section:G-NA}

As a way of providing positive results on the denseness of the set of norm-attaining  operators, Lindenstrauss introduced what is nowadays known as the \emph{Lindenstrauss property A} (see \cite{Lind63}). A Banach space $X$ is said to satisfy this property if, for every Banach space $Y$, the set of norm-attaining operators from $X$ into $Y$ is dense in $\mathcal{L}(X,Y)$. In \Cref{def:G-propertyA}, we introduced \emph{$G$-property A} as a natural equivariant version of this notion which, as mentioned in the introduction, was one of the main motivations for the present work. 

There is a clear connection between the Bishop-Phelps property and Lindenstrauss property A: a Banach space has the former if and only if it has the latter under every equivalent renorming. The reader can check that the same relationship holds between the $G$-BP and $G$-property A when one considers $G$-invariant renormings, i.e., those for which the $G$-action remains isometric with respect to the new norm. In other words, the corresponding linear isomorphism is $G$-equivariant (see e.g. \cite{AFGR, Michal} for recent research on $G$-invariant renormings).

Since the $G$-property~A is weaker than the $G$-BP, there exist more examples of $G$-Banach spaces satisfying it, as we demonstrate in this section (see \Cref{thm:ell1}, \Cref{Prop-c0-dense,prop:cocompact,prop:bigelements}, and \Cref{cor:ATimpliesG-BP} for positive results on the $G$-property A, and \Cref{prop-negative-Lind} for a negative one).

We begin by showing that $\ell_1$ has the $G$-property A for every topological group $G$ and any action. In fact, we shall prove something stronger. Recall that the extreme points of the unit ball of $\ell_1$ are precisely the canonical basis elements $\{\pm e_i\}_{i\in\N}$; hence every linear isometry on $\ell_1$ acts by permuting the canonical basis and changing their signs. In particular, every group acting linearly and isometrically on $\ell_1$ induces an action on $\N$, the index set of the basis elements.

\begin{theorem}\label{thm:ell1}
    Let $G$ be a topological group so that $G \curvearrowright \ell_1$. Then $\ell_1$ with this action of $G$ has the $G$-property A. In fact, if the induced action of $G$ on $\N$ has finitely many orbits, then every $G$-equivariant operator is norm-attaining.
\end{theorem}

\begin{proof} 
    Given the action $G\curvearrowright \ell_1$, consider the induced action of $G$ on $\N$. Let $(I_n)_{n\in\Gamma}$ be an enumeration of the orbits of this action, where $\Gamma$ is a finite or countably infinite index set. Let $Y$ be any Banach space and $T \in \mathcal{L}^G (\ell_1, Y)$. We split the proof into two cases.

	\noindent
	\textbf{Case 1} $|\Gamma|<\infty$: In this case, since $\|T\|=\sup_{n \in\mathbb{N}} \|Te_n\|$, we conclude that $T \in \NA^G (\ell_1, Y)$.

	\noindent
	\textbf{Case 2} $|\Gamma|=\infty$: We may assume that $\Gamma=\N$ and $\|T\|=1$.
	Given $\eps \in (0,1)$, choose $0 < \eta < \e$. Let $n \in \mathbb{N}$ be such that $\|Te_n\| >1-\eta/2$. For simplicity, suppose that {$\|Te_1\|>1-\eta/2$, and $1\in I_1$}. Let $P_1 : \ell_1 \to \ell_1$ be the canonical projection onto $\ell_1 (I_1)$, that is, $( P_1 (x) )(i) = x(i)$ when $i \in I_1$ and $0$ otherwise (noting that $I_1 \neq \mathbb{N})$.
	Consider $S: \ell_1 \to Y$ defined by
	\begin{equation*} 
	Sx := Tx + \eta \, T\circ P_1 (x)
	\end{equation*} 
	for every $x \in \ell_1$. Notice that $S$ is $G$-equivariant and $\|S-T\| \leq \eta < \varepsilon$. Moreover, 
	\begin{equation*}
	\|Se_i\| = \|Te_i + \eta Te_i\| > (1+\eta)\left(1-\frac{\eta}{2}\right) >1
	\end{equation*} 
	for every $i \in I_1$, while $\|Se_j\| = \|Te_j\| \leq 1$ for all $j \in \mathbb{N} \setminus I_1$. This implies that $\|S\| = \sup_{i \in I_1} \|Se_i\| = \|Se_1\|$. This shows that $S \in \NA^G(\ell_1 ,Y)$. 
\end{proof}

An application of \Cref{thm:ell1} yields the following result.

\begin{proposition} \label{Prop-c0-dense} Let $G$ be a topological group and $X$ be a reflexive $G$-Banach space. Suppose that $G$ also acts on $c_0$ linearly isometrically. Then $\NA^G(X,c_0)$ is dense in $\mathcal{L}^G(X,c_0)$.

\end{proposition} 
\begin{proof}
	Let $T \in \mathcal{L}^G(X, c_0)$ and $\eps > 0$. If we equip $X^*$ with the standard dual action, then the property of being $G$-equivariant passes to adjoint operators, so $T^* \in \mathcal{L}^G(\ell_1, X^*)$. By \Cref{thm:ell1}, there exists $S \in \NA^G(\ell_1, X^*)$ with $\|T^* - S\| < \eps$. Taking adjoints again, we get that $S^* \in \NA^G(X^{**}, \ell_\infty)$ and $\|T^{**} - S^*\| < \eps$. Since $X$ is reflexive, we identify $X$ with $X^{**}$. Thus, $S^*$ may be regarded as an operator from $X$ to $\ell_\infty$, and it satisfies $\|T - S^*\|<\e$. 
    
    Since $S \in \NA^G(\ell_1, X^*)$, there exists $w \in S_{\ell_1}$ such that $\|Sw\| = \|S\|$. Since $X$ is reflexive, the functional $Sw \in X^*$ also attains its norm on $B_X$ at some $x \in S_X$. Thus,
    \[
    |(Sw)(x)| = |(S^*x )(w)| \leq \|S^*x\|_{\infty} \leq \|S^*|_X\| \leq \|S\|.
    \]
    In particular, $\|S^*x\|_{\infty} = \|S^*\|$ and $S^*$ is norm-attaining. It remains only to show that the range of $S^* $ is contained in $c_0$.

	By inspecting the proof of \Cref{thm:ell1}, if the induced action on $\mathbb{N}$ has finitely many orbits, take $S=T^*$, and the conclusion is immediate. Otherwise, we see that $S = T^* + \eta\, T^* \circ P$, where $0 < \eta < \varepsilon$ and $P: \ell_1 \to \ell_1$ is the canonical projection onto $\overline{\mathrm{span}}\{e_i : i \in I\}$, where $I \subseteq \N$ is the set of indices $i$ such that either $e_i$ or $-e_i$ is an element of $G \cdot e_j$ for some fixed $j \in \N$. Looking at the adjoint of $S$, we get $S^* = T + \eta\, P^* \circ  T$, where $P^*: \ell_\infty \to \ell_\infty$ is the canonical projection onto the subspace $\ell_\infty(I) = \{x \in \ell_\infty : (\forall i \notin I) \ x(i) = 0\}$. However, since the range of the restricted map $P^* \restricted_{c_0} : c_0 \to \ell_\infty$ is contained in $c_0$, we conclude that the image of $S^*$ is contained in $c_0$; hence, $S^* \in \NA^G(X, c_0)$.
\end{proof}

In contrast to \Cref{Prop-c0-dense}, when we consider $c_0$ in the domain we get the following negative result.

\begin{proposition} \label{prop-negative-Lind}
For every compact group $G$ acting on $c_0$ by linear isometries, there is a $G$-Banach space $Y$ such that $\overline{\NA^G (c_0, Y)} \neq \mathcal{L}^G (c_0, Y)$.
\end{proposition}
\begin{proof}
Let $(X,\|\cdot\|)$ be a separable Banach space and $G$ be a compact group such that $\alpha: G \curvearrowright X$ by linear isometries. By \cite[Proposition 2.1]{AFGR} or \cite[Corollary 4.4 (1)]{Michal}, there exists a strictly convex renorming $Y=(X, \vertiii{\cdot})$ such that the norm is invariant under the group action $G$, that is, $\vertiii{x}=\vertiii{\alpha(g,x)}$ for every $x \in X$ and $g \in G$. Notice that, for such a renorming $Y$, the identity operator $I :(X,\|\cdot\|)\to Y=(X,\vertiii{\cdot})$ is $G$-equivariant. 

As a concrete case, take $X=c_0$ with the supremum norm. Let $Y=(c_0,\vertiii{\cdot})$ be a strictly convex space such that $\vertiii{\cdot}$ is $G$-invariant. It is well known that if a bounded linear operator from $c_0$ to $Y$ is norm-attaining, then it must be of finite rank (see, for instance, \cite[Lemma 2]{Martin2014}); so $I \in \mathcal{L}^G (c_0, Y)$ cannot be approximated by norm-attaining operators. 
\end{proof}

In order to obtain further examples of $G$-Banach spaces with $G$-property A, we recall the following definitions that will be useful in this context.

Let $M$ be a metric space and $G$ be a topological group acting continuously by isometries on $M$. By $M\sslash G$ we shall denote the space of closures of orbits of $G$ equipped with the infimum distance. That is, \[M\sslash G:=\{\overline{G\cdot x}\colon x\in M\}\] and for $\overline{G\cdot x},\overline{G\cdot y}\in M\sslash G$, we set 
\[
D(\overline{G\cdot x},\overline{G\cdot y}):=\inf_{g\in G} d_M(x,gy).
\] 
It is straightforward to verify that this is a well-defined metric. We can now introduce the following definition.

\begin{definition}   \label{definition:omega-categorical}
Let $G$ be a topological group, $X$ be a $G$-Banach space, and let $n\in\N$. We say that the action $G\curvearrowright X$ is {\it $n$-cocompact} whenever the metric space $B_X^n\sslash G$ is compact, where $B_X^n$ is endowed with the sum metric. 

A separable Banach space $X$ is said to be \emph{$\omega$-categorical} if $\iso(X) \curvearrowright X$ is $n$-cocompact for every $n\in\N$.
\end{definition}

Several remarks are in order. First, it is easy to check that one may equivalently use $S_X$ instead of $B_X$ in the definition of $n$-cocompactness (see, e.g., \cite{FeRV26}, where the former convention is adopted). 

Second, one may wonder whether the assumption of $\omega$-categoricity is unnecessarily strong, since in the applications below we only make direct use of $1$-cocompactness (although this would change when dealing with multilinear maps). However, it turns out that $1$-cocompactness of an action $G\curvearrowright X$ is equivalent to $n$-cocompactness for every $n$ under rather natural topological assumptions on $G$; see \cite{BYTs16}.

Consequently, it is not surprising that the main examples of $1$-cocompact actions arise from $\omega$-categorical Banach spaces. This, in turn, connects the $G$-property A with a rather different area of Banach space theory. Indeed, the notion of an $\omega$-categorical Banach space comes from model theory of Banach spaces where it forms one of the most interesting classes of separable Banach spaces fully determined by their first-order theory. The model-theory-free definition above comes from (the continuous version of) the Ryll-Nardzewski theorem which reformulates the original definition of $\omega$-categoricity through the actions of linear isometry groups (see \cite{BYU07}). For model theory of Banach spaces, and more generally metric structures, we refer the interested reader to \cite{ModelTheoryForMS}.

The following is a short list of examples of such Banach spaces (also called \emph{$\aleph_0$-categorical Banach spaces}). For more information about them, we send the reader to the references provided below along with the examples as well as to \cite{Kha21}.
\begin{enumerate}
    \item $\ell_2(\N)$ -folklore, see also the example below;
    \item $L_p$, for all $1\leq p<\infty$ - see \cite{ModelTheoryForMS};
    \item $C(2^\N)$ - by Henson (unpublished);
    \item the Gurarii space - see \cite{YaHe17};
    \item $L_p(L_q)$, for all $1\leq p,q<\infty$ - see \cite{HenRay11};
    \item if $1\leq p<\infty$ and $X$ is $\omega$-categorical, then so is $L_p(X)$ (\cite{FeLPOP}).
\end{enumerate}

\begin{proposition}\label{prop:cocompact}  
Let $X$ be a $G$-Banach space, for some topological group $G$, and suppose $G \curvearrowright X$ is $1$-cocompact. Then, every $G$-equivariant operator is norm-attaining. In other words, $\NA^G(X,Y) = \mathcal{L}^G(X,Y)$ for every $G$-Banach space $Y$. In particular, $X$ has the $G$-property A.
\end{proposition}

\begin{proof} Let $T \in \mathcal{L}^G (X,Y)$ be given. 
   Let us consider the map $T' : B_X \sslash G \to Y \sslash G$ given by $T' (\overline{G\cdot x}) = \overline{G \cdot Tx}$ for each $x \in B_X$. We claim that the map is well-defined. In fact, if $\overline{G\cdot x} = \overline{G\cdot y}$, then there exists a sequence $(g_n) \subseteq G$ such that $\|x-g_ny\| \to 0$. Then 
   \[\| Tx - g_n Ty\|= \|Tx-T(g_ny)\| \leq \|T\| \|x-g_ny\| \to 0.
   \] Thus $\overline{G\cdot Tx}= \overline{G\cdot Ty}$. Moreover, notice that 
    \[
    \|Tx \| = \inf_{g\in G} \|Tx-g\cdot 0\| = D(\overline{G\cdot Tx}, 0) = D(T' (\overline{G\cdot x}), 0) 
    \]
    for every $x \in B_X$; $\|T\|=\sup \{ D(T' (\overline{G\cdot x}), 0) : \overline{G\cdot x} \in B_X \sslash G\}$. If we prove that $T'$ is continuous, then the preceding supremum is indeed achieved since $B_X \sslash G$ is compact. To see that $T'$ is continuous, pick $\overline{G\cdot x}\in B_X\sslash G$ and $\varepsilon>0$. Then there is $\delta>0$ such that for every $y\in B_X$ with $\|x-y\|<\delta$ we have $\|T(x)-T(y)\|<\varepsilon$. Then if $\overline{G\cdot z}\in B_X\sslash G$ is such that $D(\overline{G\cdot x},\overline{G\cdot z})<\delta$ there is $g\in G$ so that $\|x-gz\|<\delta$ thus 
    \[
    D(T'(\overline{G\cdot x}), T'(\overline{G\cdot z})) =\inf_{h \in G} \|Tx- h Tz \| \leq \|Tx - T(gz)\| <\varepsilon. \qedhere
   \] \end{proof}

As an immediate consequence of Proposition \ref{prop:cocompact} we have that $\omega$-categorical Banach spaces satisfy the $G$-property A.

\begin{corollary}
Every $\omega$-categorical Banach space $X$ has the $G$-property A with respect to $G=\Iso(X)$.
\end{corollary}

We will prove another positive result in \Cref{prop:bigelements} for Banach spaces that contain big points. Let us recall the notions of big points, convex transitivity, and almost transitivity.

\begin{definition} \label{def:big-points}
Let $X$ be a Banach space. We say that 
\begin{enumerate}
\item[(a)] $x\in S_X$ is a \emph{big point} if $\overline{\co}\big(\Iso(X)\cdot x\big)=B_X$.
\item[(b)] $X$ is \emph{convex-transitive} if every element $x\in S_X$ is a big point. 
\item[(c)] $X$ is \emph{almost-transitive} if there is/for every $x\in S_X$ the orbit $\Iso(X)\cdot x$ is dense in $S_X$.
\end{enumerate}
\end{definition}
Clearly, every almost-transitive Banach space is convex-transitive, whereas $C(K)$, where $K$ denotes the Cantor set, is a convex-transitive Banach space that is not almost-transitive \cite{PelczynskiRolewicz1962}. For more information on convex-transitive Banach spaces and big points as well as almost-transitive Banach spaces we refer the reader to \cite{BG-RP99,BG-RP02}. We note that these properties were developed as weakenings of transitive Banach spaces related to the famous Mazur rotation problem and we refer to the survey \cite{CS-FeRa22} as a background for these connections. Now we can provide the promised result.
 
\begin{proposition}\label{prop:bigelements}
Let $X$ be a Banach space that contains a big point. Then, every $G$-equivariant operator is norm-attaining at that point, where $G=\Iso(X)$. In particular, $X$ has the $G$-property A.
\end{proposition}

\begin{proof}
Let $Y$ be a $G$-Banach space, $T \in \mathcal{L}^G (X,Y)$, and let $x \in X$ be a big point. We claim that $\|T\|=\|T(x)\|$. Indeed, for any $\varepsilon>0$, there exist $\lambda_1, \ldots, \lambda_k \geq 0$ satisfying $\sum_{i=1}^k \lambda_i=1$ and $(g_i)_{i\leq k}\subseteq G=\Iso(X)$ such that 
\begin{equation*}
    \|T\| - \e < \left\| T \left( \sum_{i=1}^k \lambda_i g_i x \right) \right\| = \left\| \sum_{i=1}^k \lambda_i g_i T(x) \right\| \leq \sum_{i=1}^k \lambda_i \|g_iT(x)\| = \|T(x)\|.
\end{equation*}

Since $\varepsilon>0$ was arbitrary, we get $\|T(x)\|\geq \|T\|$, so $\|T(x)\|=\|T\|$.
\end{proof}

For Banach spaces that are moreover almost-transitive, one can get something stronger. In particular, it applies to every $L_p$ space with $p \in \left[1,\infty\right)$, more generally, to $L_p(\Omega,\Sigma,\mu; E)$ whenever $\mu$ is a non-atomic positive measure and $E$ is almost transitive (see \cite{GrJaKa94}).

\begin{corollary}\label{cor:ATimpliesG-BP}
Let $X$ be an almost-transitive Banach space and $G=\Iso(X)$. Then, $X$ has the $G$-BP.
\end{corollary}
\begin{proof}
By \Cref{prop:bigelements}, we get that $B_X$ has the $G$-BP. So, it is enough to show that, up to scaling, there are no other convex closed bounded $G$-invariant subsets. Fix $C \subseteq X$ to be a nonempty bounded closed convex $G$-invariant subset. We will show that $C = rB_X$ for some $r \geq 0$. Denote $r := \sup_{x \in C} \norm{x}$. Then, clearly $C \subseteq rB_X$ and if $r = 0$, then $\{0\} = C = rB_X$. This means we can assume $r > 0$. Let $\eps \in (0, r)$ and find $x \in C$ such that $\norm{x} > r-\eps$. By the almost-transitivity and closedness of $C$, we have that $\norm{x}S_X \subseteq C$. By the convexity of $C$, we even have that $\norm{x}B_X \subseteq C$. Since $\eps \in (0, r)$ was arbitrary, using that $C$ is closed once again, we conclude that
	\begin{equation*}
		rB_X = \overline{\bigcup_{\eps \in (0, r)} (r-\eps)B_X} \subseteq \overline{C} = C. \qedhere
	\end{equation*}
\end{proof}

\section{Summary of the results}\label{section:summary}
For the convenience of the reader, we review all the implications we have proved between the various properties that we have considered throughout the text, as well as the classical counterparts, in Diagram~\ref{diagram-dentability} below.

\begin{figure}[h]\label{figure:summary}
    \hbox to \hsize{\hfil
        \hspace{0.65cm}

        \begin{tikzpicture}[
            >=latex,
            arrow/.style={
                -Implies,
                double equal sign distance,
                line width=0.65pt
            },
            iffarrow/.style={
                Implies-Implies,
                double equal sign distance,
                line width=0.65pt
            },
            box/.style={
                draw,
                rounded corners,
                minimum width=2.34cm,
                minimum height=0.78cm,
                align=center,
                font=\normalsize
            },
            layoutspacer/.style={
                box,
                draw=none
            }
        ]


        \matrix (m) [
            matrix of nodes,
            row sep=1.235cm,
            column sep=1.04cm
        ] {
            \node[box] (GLindA)
                {\,$G$-property A\,};
            &
            \node[box] (GBP)
                {$G$-BP};
            &
            &
            \node[box] (BP)
                {BP};
            &
            \\
            \node[box] (GKMP)
                {$G$-KMP};
            &
            \node[box] (strongGdent)
                {Strong\\$G$-dentability};
            &
            &
            \node[box] (dent)
                {Dentability};
            &
            \\

            &
            \node[box] (weakGdent)
                {Weak\\$G$-dentability};
            &
            \node[layoutspacer] (thirdColumnSpacer)
                {\phantom{Very weak}\\
                 \phantom{$G$-dentability}};
            &
            \\
            \node[box] (KMP)
                {KMP};
            &
            \node[box] (GRNP)
                {$G$-RNP};
            &
            &
            \node[box] (RNP)
                {RNP};
            &
            \\
        };


        \draw[arrow]
            (RNP) -- (GRNP);

        \draw[iffarrow]
            (BP) -- (dent);

        \draw[iffarrow]
            (dent) -- (RNP);

        \draw[arrow]
            (GBP) -- (GLindA);


        \draw[arrow, bend left=30]
            (strongGdent) to (weakGdent);

        \draw[arrow, bend left=12]
            (strongGdent)
            to node[
                below,
                midway,
                font=\tiny
            ] {$G$ compact}
            (GKMP);


        \path
            let
                \p1 = (RNP.south),
                \p2 = (KMP.south)
            in
                coordinate (dropA)
                    at (\x1,\y1-8mm)
                coordinate (turnA)
                    at (\x2,\y1-8mm);

        \draw[
            arrow,
            shorten >=1pt,
            shorten <=1pt
        ]
            (RNP.south)
            -- (dropA)
            -- (turnA)
            -- (KMP.south);


        \draw[arrow, bend left=30]
            (KMP) to (GKMP);

        \draw[arrow, bend left=30]
            (GBP)
            to node[
                midway,
                right,
                font=\tiny
            ] {$G$ compact}
            (strongGdent);

        \draw[iffarrow, bend left=10]
            (dent) to (weakGdent.east);


        \draw[arrow, bend left=0]
            (GRNP)
            to node[
                midway,
                left,
                font=\tiny
            ] {
                \parbox{1.7cm}{
                    $G$ loc. cpt and $\sigma$-cpt
                }
            }
            (weakGdent);


        \draw[dotted, arrow, bend left=6]
            (GKMP) to (strongGdent);

        \draw[dotted, arrow, bend left=30]
            (GKMP) to (KMP);

        \draw[dotted, arrow, bend left=30]
            (strongGdent) to (GBP);

        \draw[dotted, arrow, bend left=10]
            (BP) to (GBP);




        \path
            let
                \p1 = (GRNP.south),
                \p2 = (RNP.south)
            in
                coordinate (dropGR)
                    at (\x1,\y1-3.5mm)
                coordinate (turnGR)
                    at (\x2-4mm,\y1-3.5mm);

        %



        \draw[arrow, bend left=0]
            (dent) to (strongGdent);

        \path[bend left=0]
            (dent)
            to node[
                midway,
                sloped,
                font=\Large,
                inner sep=0pt
            ] {$\times$}
            (strongGdent);


        \draw[arrow, bend left=10]
            (GBP) to (BP);

        \path[bend left=10]
            (GBP)
            to node[
                midway,
                sloped,
                font=\Large,
                inner sep=0pt,
                xshift=1pt
            ] {$\times$}
            (BP);


        \draw[arrow, bend left=30]
            (weakGdent) to (strongGdent);

        \path[bend left=30]
            (weakGdent)
            to node[
                midway,
                sloped,
                font=\Large,
                inner sep=0pt,
                xshift=1pt
            ] {$\times$}
            (strongGdent);

        \end{tikzpicture}

    \hfil}

    \captionsetup{
        labelformat=simple,
        labelsep=colon,
        name=Diagram
    }

    \caption{
        This diagram summarizes all the implications among the
        properties we have considered in the paper.
        Dotted arrows denote cases where it remains unclear whether
        the implication holds, whereas crossed arrows mark implications
        that generally fail to hold; counterexamples are provided
        throughout the text.
    }

    \label{diagram-dentability}
\end{figure}

 The reader can see from the diagram that the strongest results are obtained for $G$-Banach spaces when $G$ is compact. In this case, $G$-BP implies all the other $G$-properties.

 \begin{fact}
We briefly explain the crossed arrows, which indicate implications that fail to hold in general, by recalling the counterexamples provided earlier.
\begin{enumerate}
        \itemsep0.25em
        \item $L_1$ has, for the isometry group $G=\Iso(L_1)$, the $G$-BP by \Cref{cor:ATimpliesG-BP}. However, $L_1$ clearly fails the BP. This shows that the implication
    \[
    G\text{-BP} \implies \mathrm{BP}
    \]
    does not hold in general. Nevertheless, it does hold when $G$ is compact.
        \item Moreover, for $G=\Iso(L_1)$, the space $L_1$ fails to have the $G$-RNP by \Cref{ex:L1-RNP}. It is also not weakly $G$-dentable. Indeed, since $L_1$ is almost transitive, the proof of \Cref{cor:ATimpliesG-BP} shows that every closed, bounded, convex $G$-invariant subset of $L_1$ is a closed ball centered at $0$, which is not dentable.
        \item We observed in \Cref{cor:strongGdentIsStrong} that $L_p$ is not strongly $G$-dentable with respect to $G=\Iso (L_p)$, so the implication
    \[
    \text{Weak $G$-dentability (Dentability)} \implies \text{Strong $G$-dentability}
    \]
    does not hold in general. 
    \end{enumerate} 
 \end{fact}

\begin{problem} Decide whether the implications of the dotted arrows hold true.
\end{problem}

One of the most mysterious implications is the one between the classical Bishop-Phelps property and the $G$-Bishop-Phelps property. Notice that it is the only instance of the phenomenon, among the properties we have considered here, that the classical property does not obviously imply its $G$-analogue. It is possibly also related to the apparent difficulty of showing that strong $G$-dentability implies the $G$-BP (for a compact group $G$). It would certainly be very interesting to find an example of a $G$-Banach space with the BP which is strongly $G$-dentable but fails the $G$-BP, thereby disproving both implications at once.

It is not surprising that the question whether the $G$-KMP implies strong $G$-dentability remains open, given that even in the classical setting the long-standing problem of whether the KMP implies the RNP is still unresolved (see, for instance, \cite{LopezPerezMedina2024, Lopez-PerezMena1996}). This indicates that the $G$-analogue will likewise require further investigation, a direction we hope to explore in future work.

\addtocontents{toc}{\protect\setcounter{tocdepth}{1}}
\subsection*{Acknowledgments} 
The authors would like to thank Valentin Ferenczi for his advice on some references. The authors also thank the anonymous referee for carefully reading the manuscript, identifying inaccuracies in the original definition of the $G$-Radon-Nikod\'ym property, and providing a number of comments
which have improved the paper.

Sheldon Dantas was supported by PID2021-122126NB-C33 funded by MICIU/AEI \\ 
/10.13039/501100011033 and by ERDF/EU. Michal Doucha and Tom\'a\v{s} Raunig were supported by the GAČR project 25-15366S and by the Czech Academy of Sciences (RVO 67985840). Mingu Jung was supported by June E Huh Center for Mathematical Challenges (HP086601) at Korea Institute for Advanced Study and by the research fund of Hanyang University (HY-202500000003346).

\bibliographystyle{abbrv}

\bibliography{references}

\end{document}